\documentclass[a4paper,draft]{amsart}
\usepackage[latin1]{inputenc}  
\usepackage{amsmath}  
\usepackage{amssymb}  
\usepackage{curves}  
\usepackage{epic}  
\usepackage[all]{xy}   
\usepackage{graphicx}   
\usepackage{euscript}
\usepackage[english]{babel}
\usepackage{rotating}

\newtheorem{thm}{Theorem}[section]
\newtheorem{prop}[thm]{Proposition}

\newtheorem{cor}[thm]{Corollary}
\newtheorem{lem}[thm]{Lemma}

\newtheorem{defn}[thm]{Definition}

\newtheorem{remark}[thm]{Remark}

\newtheorem{example}[thm]{Example}
\newtheorem{examples}[thm]{Examples}

\newenvironment{steps}[1]{%
\refstepcounter{thm}\noindent{\bf (\thethm)\ {#1}}}
{\em}

      \newcommand{\N}{{\mathbb N}}
\newcommand{\Z}{{\mathbb Z}}      \newcommand{\R}{{\mathbb R}}
      \newcommand{\C}{{\mathbb C}}
     
\newcommand{\HH}{{\mathbb H}}   

      \newcommand{\Cont}{{\mathcal C}}

\newcommand{\LIP}{\operatorname{LIP}}
\newcommand{\Lip}{\operatorname{Lip}}

\newcommand{\Mod}{\operatorname{Mod}}

\newcommand{\veps}{\varepsilon}
\newcommand{\jint}{\int\hspace{-3.8mm}\frac{\,\,\,}{}}
\newcommand{\jintb}{\int\hspace{-2.75mm}\frac{\,\,}{}}
\newcommand{\fin}{\hspace*{\fill}$\square$\\ }
\newcommand{\rotup}{\begin{rotate}{45}\hspace*{-3mm}$\subsetneq$\end{rotate}}
\newcommand{\rotdown}{\begin{rotate}{-45}\hspace*{-3mm}$\subsetneq$\end{rotate}}
\newcommand{\rotupa}{\begin{rotate}{45}$\subsetneq$\end{rotate}}
\newcommand{\rotdowna}{\begin{rotate}{-45}$\subsetneq$\end{rotate}}


\begin{document}
\title{infinitesimally Lipschitz functions on metric spaces}
\author{E. Durand and J.A. Jaramillo}
\address{Departamento de An{\'a}lisis Matem{\'a}tico, Universidad Complutense de Madrid, 28040 Madrid, Spain}
\email{estibalitzdurand@mat.ucm.es} \email{jaramil@mat.ucm.es}
\thanks{Research partially supported by DGES (Spain) MTM2006-03531.}

\begin{abstract}
For a metric space $X$, we study the space $D^{\infty}(X)$ of
bounded functions on $X$ whose infinitesimal Lipschitz constant is
uniformly bounded. $D^{\infty}(X)$ is compared with the space
$\LIP^{\infty}(X)$ of bounded Lipschitz functions on $X$, in terms
of different properties regarding the geometry of $X$. We also
obtain a Banach-Stone theorem in this context. In the case of a
metric measure space, we also compare $D^{\infty}(X)$ with the
Newtonian-Sobolev space $N^{1, \infty}(X)$. In particular, if $X$
supports a doubling measure and satisfies a local Poincar{\'e}
inequality, we obtain that $D^{\infty}(X)=N^{1, \infty}(X)$.
\end{abstract}
\maketitle

\section{Introduction}
Recent years have seen many advances in geometry and analysis, where
first order differential calculus has been extended to the setting
of spaces with no a priori smooth structure; see for instance
\cite{Amb,hei,He2,Sem}. The notion of derivative measures the
infinitesimal oscillations of a function at a given point, and gives
information concerning for instance monotonocity. In general metric
spaces we do not have a derivative, even in the weak sense of
Sobolev spaces. Nevertheless, if $f$ is a real-valued function on a
metric space $(X,d)$ and $x$ is a point in $X$, one can use similar
measurements of sizes of first-order oscillations of $f$ at small
scales around $x$, such as
$$
D_{r}f(x)=\frac{1}{r}\sup\Big\{|f(y)-f(x)|:y\in X, d(x,y)\leq r\Big
\}.
$$
On one hand, this quantity does not contain as much information as
standard derivatives on Euclidean spaces does (since we omit the
signs) but, on the other hand, it makes sense in more general
settings since we do not need any special behavior of the underlying
space to define it. In fact, if we look at the superior limit of the
above expression as $r$ tends to $0$ we almost recover in many
cases, as in the Euclidean or Riemannian setting, the standard
notion of derivative. More precisely, given a continuous function $f:X\to\R$,
the \em infinitesimal Lipschitz constant \em at a point $x\in X$ is
defined as follows:
$$
\Lip f(x)=\limsup_{r\to 0}D_r f(x)=\limsup_{\substack{y\to x\\ y\neq
x}} \frac{|f(x)-f(y)|}{d(x,y)}.
$$
Recently, this functional has played an important role in several
contexts. We just mention here the construction of differentiable
structures in the setting of metric measure spaces \cite{Che,Kei} ,
the theory of upper gradients \cite{HK,Sh1}, or the Stepanov's
differentiability theorem \cite{BRZ}.

This concept gives rise to a class of function spaces, \em
infinitesimally Lipschitz function spaces\em, which contains in some
sense infinitesimal information about the functions,
$$D(X)=\{f:X\longrightarrow\R:\, \|\Lip
f\|_{\infty}<+\infty\}.$$

This space $D(X)$ clearly contains the space $\LIP(X)$ of Lipschitz
function and a first approach should be comparing such spaces. In
Corollary \ref{main2} we give sufficient conditions on  the metric
space $X$ to guarantee the equality between  $D(X)$ and $\LIP(X)$. A
powerful tool which transforms bounds on infinitesimal oscillation
to bounds on maximal oscillation is a kind of mean value theorem
(see Lemma $2.5$ in \cite{Sem}). In fact, the largest class of
spaces for which we obtain a positive answer is the class of
quasi-length spaces, which has a characterization in terms of such
mean type value theorem. In particular, this class includes
quasi-convex spaces. In addition, we present some examples for which
$\LIP(X)\neq D(X)$ (see Examples \ref{ejem} and \ref{cuspide}).

At this point, it seems natural to approach the problem of
determining which kind of spaces can be classified by their
infinitesimal Lipschitz structure. Our strategy will be to follow
the proof in \cite{GaJa2} where the authors find a large class of
metric spaces for which the algebra of bounded Lipschitz functions
determines the Lipschitz structure for $X$. A crucial point in the
proof is the use of the Banach space structure of $\LIP(X)$. Thus,
we endow $D(X)$ with a norm  which arises naturally from the
definition of the operator $\Lip$. This norm is not complete  in the
general case, as it can be seen in Example \ref{nobanach}. However,
there is a wide class of spaces, the locally radially quasiconvex
metric spaces (see Definition \ref{locradqua}), for which
$D^{\infty}(X)$ (bounded infinitesimally Lipschitz functions) admits
the desired Banach space structure. Moreover, for such spaces, we
obtain a kind of Banach-Stone theorem in this framework (see Theorem
\ref{banachstone}).

If we have a measure on the metric space, we can deal with many more
problems. In this line, there are for example generalizations of
classical Sobolev spaces to the setting of arbitrary metric measure
spaces. It seems that Hajs{\l}az  was the first who introduced
Sobolev type spaces in this context \cite{haj1}. He defined the
spaces $M^{1,p}(X)$ for $1\leq p\leq \infty$ in connection with
maximal operators. It is well known that $M^{1,\infty}(X)$ is in
fact the space of bounded Lipschitz functions on $X$. Shanmugalingam
in \cite{Sh1} introduced, using the notion of upper gradient (and
more generally weak upper gradients) the Newtonian spaces
$N^{1,p}(X)$ for $1\leq p<\infty$. The generalization to the case
$p=\infty$ is straightforward and we will compare the function
spaces $D^{\infty}(X)$ and $\LIP^{\infty}(X)$ with such Sobolev
space, $N^{1,\infty}$. From Cheeger's work \cite{Che}, metric spaces
with a doubling measure and a Poincaré inequality admit a
differentiable structure with which Lipschitz functions can be
differentiated almost everywhere. Under the same hypotheses we prove
in Corollary \ref{NLip} the equality of all the mentioned spaces.
Furthermore, if we just require a local Poincaré inequality we
obtain $M^{1,\infty}(X)\subseteq D^{\infty}(X)= N^{1,\infty}(X)$.
For further information about more generalizations of Sobolev spaces
on metric measure spaces see \cite{haj}.

We organized the work as follows. In Section $2$ we will introduce
\em infinitesimally Lipschitz function spaces \em $D(X)$ and we look
for conditions regarding the geometry of the metric spaces we are
working with in order to understand in which cases the infinitesimal
Lipschitz information yields the global Lipschitz behavior of a
function. Moreover, we show the existence of metric spaces for which
$\LIP(X)\subsetneq D(X)$. In Section $3$ we introduce the class of
\em locally radially quasiconvex metric spaces \em and we prove that
the space of bounded infinitesimally Lipschitz function can be
endowed with a natural Banach space structure. The purpose of
Section $4$ is to state a kind of Banach-Stone theorem in this
context while the aim of Section $5$ is to compare the function
spaces $D^{\infty}(X)$ and $\LIP^{\infty}(X)$ with Sobolev spaces in
metric measure spaces.

\section{infinitesimally Lipschitz functions}
Let $(X,d)$ be a metric space. Given a function $f:X\to\R$, the \em
infinitesimal Lipschitz constant \em of $f$ at a non isolated point
$x\in X$ is defined as follows:
$$
\Lip f(x)=\limsup_{\substack{y\to x\\ y\neq x}}
\frac{|f(x)-f(y)|}{d(x,y)}.
$$
If $x$ is an isolated point we define $\Lip f(x)=0$. This value is
also known as \em upper scaled oscillation \em (see \cite{BRZ}) or
as \em pointwise infinitesimal Lipschitz number \em (see
\cite{He2}).

\begin{examples}\em
$(1)$ If $f\in C^1(\Omega)$ where $\Omega$ is an open subset of
Euclidean space, or of a Riemannian manifold, then $ \Lip f=|\nabla
f|.$

\vspace{2mm} $(2)$ Let $\HH$ be the first Heisenberg group, and
consider an open subset $\Omega\subset\mathbb{H}$. If $f\in
C_{H}^1(\Omega)$, that is, $f$ is $H-$continu\-ously differentiable
in $\Omega$, then $ \Lip f=|\nabla_{H}f|$ where $\nabla_{H}f$
denotes the horizontal gradient of $f$. For further details see
\cite{Ma}.

\vspace{2mm} (3) If $(X,d,\mu)$ is a metric measure space which
admits a measurable differentiable structure $\{(X_{\alpha},{\bf
x}_{\alpha})\}_{\alpha}$ and $f\in\LIP(X)$, then $ \Lip
f(x)=|d^{\alpha}f(x)|\,\, \text{$\mu-$a.e.},$ where $d^{\alpha}f$
denotes the Cheeger's differential. For further information about
measurable differentiable structures see \cite{Che,Kei}.
\end{examples}

Loosely speaking, the operator $\Lip f$ estimates some kind of
infinitesimal lipschitzian property around each point. Our first aim
is to see under which conditions a function $f:X\to\R$ is Lipschitz
if and only if $\Lip f$ is a bounded functional. It is clear that if
$f$ is a $L-$Lipschitz function, then $\Lip f(x)\leq L$ for every
$x\in X$. More precisely, we consider the following spaces of
functions:
\begin{itemize}\item[$\diamond$] $\LIP(X)=\{f:X\longrightarrow\R: f  \text{ is
Lipschitz}\}$\item[$\diamond$] $D(X)=\{f:X\longrightarrow\R:\,
\sup_{x\in X}\Lip f(x)=\|\Lip f\|_{\infty}<+\infty\}.$\end{itemize}

We denote by $\LIP^{\infty}(X)$ (respectively $D^{\infty}(X)$) the
space of bounded Lipschitz functions (respectively, bounded
functions which are in $D(X)$) and $\Cont(X)$ will denote the space
of continuous functions on   $X$. It is not difficult to see that
for $f\in D(X)$, $\Lip f$ is a Borel function on $X$ and that
$\|\Lip (\cdot)\|_{\infty}$ yields a seminorm in $D(X)$. In what
follows, $\|\cdot\|_{\infty}$ will denote the supremum norm whereas
$\|\cdot\|_{L^{\infty}}$ will denote the essential supremum norm,
provided we have a measure on $X$. In addition, $\LIP(\cdot)$ will
denote the Lipschitz constant.

Since functions with uniformly bounded infinitesimal Lipschitz
constant  have a flavour of differentiability it seems reasonable to
determine if the infinitesimally Lipschitz functions are in fact
continuous. Namely,

\begin{lem}\label{dcont}
Let $(X,d)$ be a metric space. Then $D(X)\subset\Cont(X)$. \end{lem}
\begin{proof}
Let $x_0\in X$ be a non isolated point and $f\in D(X)$. We are going
to see that $f$ is continuous at $x_0$. Since $f\in D(X)$ we have
that $\|\Lip f\|_{\infty}=M<\infty$, in particular, $\Lip f(x_0)\leq
M$. By definition we have that
$$
\Lip f(x_0)=\inf_{r>0}\sup_{\substack{d(x_0,y)\leq r\\ y\neq
x_0}}\frac{|f(x_0)-f(y)|}{d(x_0,y)}.
$$
Fix $\varepsilon>0$. Then, there exists $r>0$ such that
$$
\frac{|f(x_0)-f(z)|}{d(x_0,z)}\leq\sup_{\substack{d(x_0,y)\leq r\\
y\neq x_0}}\frac{|f(x_0)-f(y)|}{d(x_0,y)}\leq
M+\varepsilon\qquad\forall z\in B(x_0,r),
$$
and so
$$
|f(x_0)-f(z)|\leq (M+\varepsilon)d(x_0,z)\qquad\forall z\in
B(x_0,r).
$$
Thus, if $d(x_0,z)\rightarrow 0$ then $|f(x_0)-f(z)|\rightarrow 0$,
and so
 $f$ is continuous at $x_0$.
\end{proof}

Now we look for conditions regarding the geometry of the metric
space $X$under which $\LIP(X)=D(X)$ (respectively
$\LIP^{\infty}(X)=D^{\infty}(X)$). As it can be expected, we need
some kind of \em connectedness\em. In fact, we are going to obtain a
positive answer in the class of \em length spaces \em or, more
generally, of \em quasi-convex spaces\em. Recall that the \em length
\em of a continuous curve $\gamma:[a,b]\rightarrow X$ in a metric
space $(X,d)$ is defined as
$$
\ell(\gamma)=\sup\Big\{\sum_{i=0}^{n-1}d(\gamma(t_i),\gamma(t_{i+1}))\Big\}
$$
where the supremum is taken over all partitions
$a=t_0<t_1<\cdots<t_n=b$ of the interval $[a,b]$. We will say that a
curve $\gamma$ is es \em rectifiable \em if $\ell(\gamma)<\infty$.
Now, $(X,d)$ is said to be a \em length space \em if for each pair
of points $x,y\in X$ the distance $d(x,y)$ coincides with the
infimum of all lengths of  curves in $X$ connecting $x$ with $y$.
Another interesting class of metric spaces, which contains length
spaces, are the so called \em quasi-convex \em spaces. Recall that a
metric space $(X,d)$ is \em quasi-convex \em if there exists a
constant $C>0$ such that for each pair of points $x,y\in X$, there
exists a curve $\gamma$ connecting $x$ and $y$ with
$\ell(\gamma)\leq Cd(x,y).$ As one can expect, a metric space is \em
quasi-convex \em if, and only if, it is bi-Lipschitz homeomorphic to
some length space.

We begin our analysis with a technical result.

\begin{lem}\label{help}
Let $(X,d)$ be a metric space and let $f\in D(X)$. Let $x,y\in X$
and suppose that there exists a rectifiable curve $\gamma:[a,b]\to
X$ connecting $x$ and $y$, that is, $\gamma(a)=x$ and $\gamma(b)=y$.
Then, $|f(x)-f(y)|\leq \|\Lip f\|_{\infty}\,\ell(\gamma).$
\end{lem}

\begin{proof}
Since $f\in D(X)$, we have that $M=\|\Lip f\|_{\infty}<+\infty.$ Fix
$\veps>0$. For each $t\in[a,b]$ there exists $\rho_{t}>0$ such that
if $z\in B(\gamma(t),\rho_t)\setminus\{\gamma(t)\}$ then
$$
|f(\gamma(t))-f(z)|\leq (M+\veps)d(\gamma(t),z).
$$
Since $\gamma$ is continuous, there exists $\delta_t>0$ such that
$$
I_t=(t-\delta_t,t+\delta_t)\subset\gamma^{-1}(B(\gamma(t),\rho_t)).
$$
The family of intervals $\{I_t\}_{t\in[a,b]}$ is an open covering of
$[a,b]$ and by compactness it admits a finite subcovering which will
be denote by $\{I_{t_i}\}_{i=0}^{n+1}$. We may assume, refining the
subcovering if necessary, that an interval $I_{t_i}$ is not
contained in $I_{t_j}$ for $i\neq j$. If we relabel the indices of
the points $t_i$ in non-decreasing order, we can now choose a point
$p_{i,i+1}\in I_{t_i}\cap I_{t_{i+1}}\cap(t_i,t_{i+1})$  for each
$1\leq i\leq n-1$. Using the auxiliary points that we have just
chosen, we deduce that:
$$
d(x,\gamma(t_1))+\sum_{i=1}^{n-1}\Big[d(\gamma(t_i),\gamma(p_{i,i+1}))+d(\gamma(p_{i,i+1}),\gamma(t_{i+1}))\Big]+d(\gamma(t_n),y)\leq\ell(\gamma),
$$
and so $|f(x)-f(y)|\leq (M+\veps) \ell(\gamma)$. Finally, since this
is true for each $\veps>0$, we conclude that $|f(x)-f(y)|\leq \|\Lip
f\|_{\infty}\,\ell(\gamma)$, as wanted.
\end{proof}

As a straightforward consequence of the previous result, we deduce
\begin{cor}\label{lengthspace}
If $(X,d)$ is a quasi-convex space then $\LIP(X)=D(X)$.
\end{cor}

The proof of the previous result is based on the existence of curves
connecting each pair of points in $X$ and whose length can be
estimated in terms of the distance between the points. A reasonable
kind of spaces in which we can approach the problem of determining
if $\LIP(X)$ and $D(X)$ coincide, are the so called \em chainable
spaces\em. It is an interesting class of metric spaces containing
length spaces and quasi-convex spaces. Recall that a metric space
$(X,d)$ is said to be \em well-chained \em or \em chainable \em if
for every pair of points $x,y\in X$ and for every $\varepsilon>0$
there exists an $\varepsilon-$chain joining $x$ and $y$, that is, a
finite sequence of points $z_1=x,z_2,\ldots,z_\ell=y$ such that
$d(z_i,z_{i+1})<\varepsilon$, for $i=1,2,\ldots,\ell-1$. In such
spaces there exist ``chains'' of points which connect two given
points, and for which the distance between the \em nodes\em, which
are the points $z_1,z_2,\ldots,z_\ell$, is arbitrary small. However,
throughout some examples we will see that there exists chainable
spaces for which the spaces of functions $\LIP(X)$ and $D(X)$ do not
coincide (see Example \ref{ejem}). Nevertheless, if we work with a
metric space $X$ in which we can control the number of nodes in the
chain between two points in terms of the distance between that
points, then we will obtain a positive answer to our problem. A
chainable space for which there exists a constant $K$ (which only
depends on $X$) such that for every $\varepsilon>0$ and for every
$x,y\in X$ there exists an $\varepsilon-$chain
$z_1=x,z_2,\ldots,z_\ell=y$ such that
$$
(\ell-1)\varepsilon\leq K(d(x,y)+\varepsilon)
$$
is called a \em quasi-length space\em. In  Lemma 2.5. \cite{Sem},
Semmes gave a characterization of quasi-length spaces in terms of a
condition which reminds a kind of ``mean value theorem''.
\begin{lem} \label{casilon}
A metric space $(X,d)$ is a quasi-length space if and only if there
exists a constant $K$ such that for each $\varepsilon>0$ and each
function $f:X\longrightarrow\R$ we have that
$$
|f(x)-f(y)|\leq K(d(x,y)+\varepsilon)\sup_{z\in
X}D_{\varepsilon}f(z)
$$
for each $x,y\in X$, where
$$
D_{\varepsilon}f(z)=\frac{1}{\varepsilon}\sup\Big\{|f(y)-f(z)|:y\in
X, d(z,y)\leq \varepsilon\Big \}.
$$
\end{lem}
The previous characterization allows us to give a positive answer to
our problem for quasi-length spaces. More precisely, we have the
following:

\begin{cor}\label{main2}
Let $(X,d)$ be a quasi-length space. Then, $ \LIP(X)=D(X).$
\end{cor}

\begin{proof}
We have to check that $D(X)\subset\LIP(X)$. Let $f\in D(X)$ and
denote by $M=\|\Lip f\|_{\infty}<+\infty$. Since $X$ is a
quasi-length space we obtain, aplying Lemma \ref{casilon}, that
there exists a constant $K\geq 1$ such that
$$
|f(x)-f(y)|\leq K(d(x,y)+\varepsilon)\sup_{z\in
X}D_{\varepsilon}f(z)
$$
for each $x,y\in X$ and each $\varepsilon>0$. Thus, if we take the superior limit when $\varepsilon$
tends to zero we deduce:
$$
|f(x)-f(y)|\leq K\, d(x,y)\sup_{z\in X}\Lip f(z)= K\,M\,d(x,y)
$$
for each $x,y\in X$. Thus, $f$ is a $KM-$Lipschitz function and we
are done.
\end{proof}

We will see in \ref{converse} that the converse of Corollary
\ref{main2} is true under more restrictive hypothesis.

Next, let us see that there exist metric spaces for which
$\LIP(X)\subsetneq D(X)$. We will approach this by constructing two
metric spaces for which $\LIP^{\infty}(X)\neq D^{\infty}(X)$. In the
first example we see that the equality fails ``for large distances''
while in the second one it fails ``for infinitesimal distances''.
\begin{example}\em\label{ejem}
Define $X=[0,\infty)=\bigcup_{n\geq 1}[n-1,n],$ and write
$I_n=[n-1,n]$ for each $n\geq 1$. Consider the sequence of functions
$f_n:[0,1]\rightarrow\R$ given by
$$
f_n (x)=\left\{\begin {array}{ll} x&\text{if}\,\,x\in\big[0,\frac{1}{n}\big]\\[10pt]
\frac{nx+n-1}{n^2}&\text{if}\,\,x\in\big[\frac{1}{n},1\big].
\end{array}
\right.
$$
For each pair of points $x,y\in I_n$, we write $
d_n(x,y)=f_n(|x-y|), $ and we define a metric on $X$ as follows.
Given a pair of points $x,y\in X$ with $x<y$, $x\in I_n$, $y\in I_m$
we define
$$
d(x,y)=\left\{\begin {array}{ll} d_n(x,y)\hspace{-5.4cm}&\text{if}\,\,n=m\\[10pt]
d_n(x,n)+\sum_{i=n+1}^{m-1}d_i(i-1,i)+d_m(m-1,y)&\text{if}\,\,n< m
\end{array}
\right.
$$
A straightforward computation shows that $d$ is in fact a metric and
it coincides locally with the Euclidean metric $d_e$. More
precisely, $$\text{ if } x\in I_n, \text{ on }
J^x=\big(x-\frac{1}{n+1}, x+\frac{1}{n+1}\big)\text{ we have that }
d|_{J^x}=d_e|_{J^x}.
$$
Next, consider the bounded function $g:X\rightarrow\R$ given by
$$
g(x)=\left\{
\begin{array}{ll}
{2k}-x&\text{if $x\in I_{2k}$,}\\[4pt]
x-{2k}&\text{if $x\in I_{2k+1}$.}
\end{array}
\right.
$$
Let us check that $g\in D^{\infty}(X)\setminus\LIP^{\infty}(X)$.
Indeed, let $x\in X$ and assume that there exists $n\geq 1$ such
that $x\in I_n$. Then, we have that if $y\in J^x$,
$$
\Lip f(x)=\limsup_{\substack{y\to x\\ y\neq x}}
\frac{|g(x)-g(y)|}{d(x,y)}=\limsup_{\substack{y\to x\\ y\neq
x}}\frac{|x-y|}{|x-y|}=1.
$$
Therefore, $g\in D^{\infty}(X)$.

On the other hand, for each positive integer $n$ we have
$|g(n-1)-g(n)|=1$ and $d(n-1,n)=f_n(1)=\frac{2n-1}{n^2}$. Thus, we
obtain that
$$
\lim_{n\rightarrow\infty}\frac{|g(n-1)-g(n)|}{d(n-1,n)}=\lim_{n\rightarrow\infty}\frac{1}{\frac{2n-1}{n^2}}=\infty
$$
and so $g$ is not a Lipschitz function.

 In particular, since
$\LIP(X)\neq D(X)$, we deduce by Corollary \ref{main2} that $X$ is
not a quasi-convex space. However, it can be checked that $X$ is a
chainable space.\fin

\begin{example}\em\label{cuspide}
Consider the set
$$
X=\{(x,y)\in\R^2:y^3=x^2, -1\leq x\leq 1\}=\{(t^3,t^2),-1\leq t\leq
1\},
$$
and let $d$ be the restriction to $X$ of the Euclidean metric of
$\R^2$. We define the bounded function
$$
g:X\rightarrow\R,\,\,(x,y)\mapsto g(x,y)=\left\{
\begin{array}{rl}
y&\text{if $x\geq 0$,}\\[4pt]
-y&\text{if $x\leq 0$.}
\end{array}
\right.
$$
Let us see that $g\in D^{\infty}(X)\backslash\LIP^{\infty}(X)$.

Indeed, if $t\neq 0$, it can be checked that $\Lip g(t^3,t^2)\leq
1$. On the other hand, at the origin we have
$$
\Lip g(0,0)=\limsup_{(x,y)\to (0,0)}
\frac{|g(x,y)-g(0,0)|}{d((x,y),(0,0))}=\limsup_{t\to
0}\frac{t^2}{\sqrt{(t^3)^2+(t^2)^2}}=1.
$$
Thus, we obtain that $\|\Lip f\|_{\infty}=1$ and so $g\in
D^{\infty}(X)$. Take now two symmetric points from the cusp with
respect to the $y-$axis, that is, $A_t=(t^3,t^2)$ and
$B_t=(-t^3,t^2)$ for $0<t<1$. In this case, we get $d(A_t,B_t)=2t^3$
and $|f(A_t)-f(B_t)|=t^2-(-t^2)=2t^2$. If $t$ tends to $0$, we have
$$
\lim_{t\to 0^{+}}\frac{|f(A_t)-f(B_t)|}{d(A_t,B_t)}=\lim_{t\to
0^+}\frac{2t^2}{2t^3}=\lim_{t\to 0^+}\frac{1}{t}= +\infty.
$$
Thus, $g$ is not a Lipschitz function.  \fin
\end{example}

\noindent In general, if $X$ is non compact space we have that
$$
\begin{array}{ccccccc}
&&&&\LIP_{\text{loc}}(X)\\[-9pt]&&&\rotup&&\rotdowna\\\LIP(X)&\subsetneq&\LIP_{\text{loc}}(X)\cap
D(X)\ &&& &\ \Cont(X)\\&&&\rotdown&&\rotupa\\[-9pt] &&&&D(X)
\end{array}
$$where $\LIP_{\text{loc}}(X)$ denotes the space of locally Lipschitz
functions. Recall that in \ref{ejem} we have constructed a function
$f\in \LIP_{\text{loc}}(X)\cap D(X)\setminus\LIP(X)$. In addition,
there is no inclusion relation between $\LIP_{\text{loc}}(X)$ and $
D(X)$. Indeed, consider for instance the metric space
$X=\bigcup_{i=1}^{\infty}B_i\subset\R$ with the Euclidean distance
where $B_i=B(i,1/3)$ denotes the open ball centered at $(i,0)$ and
radius $1/3$. One can check that the function $f(x)=ix$ if $x\in
B_i$ is locally Lipschitz whereas $f\notin D(X)$ because $\|\Lip
f\|_{\infty}=\infty$. On the other hand, the function $g$ in Example
\ref{cuspide} belongs to $D(X)\backslash\LIP_{\text{loc}}(X)$.

\section{A Banach space structure for infinitesimally Lipschitz functions}
In this section we search for sufficient conditions to have a
converse for Corollary \ref{main2}. We begin introducing a kind of
metric spaces which will play a central role throughout this
section. In addition, for such spaces, we will endow the space of
functions $D^{\infty}(X)$ and $D(X)$ with a Banach structure.

\begin{defn}\em\label{locradqua}
Let $(X,d)$ be a metric space. We say that $X$ is \em locally
radially quasi-convex \em if for each $x\in X$, there exists a
neighborhood $U^x$ and a constant $K_x>0$ such that for each $y\in
U^x$ there exists a rectifiable curve $\alpha$ in $U^x$ connecting
$x$ and $y$ such that $\ell(\gamma)\leq K_xd(x,y)$.
\end{defn}

Note that the spaces introduced in the Examples \ref{ejem} and
\ref{cuspide} are locally radially quasi-convex. Observe that there
exist locally radially quasi-convex spaces which are not locally
quasi-convex. Indeed, let
$X=\bigcup_{n=1}^{\infty}\big\{(x,\frac{x}{n}):x\in\R\big\}$ and $d$
be the restriction to $X$ of the Euclidean metric of $\R^2$. It can
be checked that $(X,d)$ is locally radially quasi-convex but it is
not locally quasi-convex.

Next, we endow the space $D^{\infty}(X)$ with the following norm:
$$
\|f\|_{D^{\infty}}=\max\{\|f\|_{\infty},\|\Lip f\|_{\infty}\}
$$
for each $f\in D^{\infty}(X)$.

\begin{thm}\label{sol1}
Let $(X,d)$ be a locally radially quasi-convex metric space. Then,
$(D^{\infty}(X),\|\cdot\|_{D^{\infty}})$ is a Banach space.
\end{thm}

\begin{proof}
Let $\{f_n\}_n$ be a Cauchy sequence in
$(D^{\infty}(X),\|\cdot\|_{D^{\infty}})$. Since $\{f_n\}_n$ is
uniformly Cauchy, there exists $f\in\Cont(X)$ such that $f_n\to f$
with the norm $\|\cdot\|_\infty$. Let us see that $f\in D(X)$ and
that $\{f_n\}_n$ converges to $f$ with respect to the seminorm
$\|\Lip (\cdot)\|_{\infty}$.

Indeed, let $x\in X$. Since $(X,d)$ is locally radially
quasi-convex, there exist a neighborhood $U^x$ and a constant
$K_x>0$ such that for each $y\in U^x$ there exists a rectifiable
curve $\gamma$ which connects $x$ and $y$ such that
$\ell(\gamma)\leq K_xd(x,y)$. By Lemma \ref{help}, we find that for
each $y\in U^x$ and for each $n,m\geq 1$
$$
|f_n(x)-f_m(x)-(f_n(y)-f_m(y))|\leq
\|\Lip(f_n-f_m)\|_{\infty}K_xd(x,y).
$$

Let $r>0$ be such that $B(x,r)\subset U_x$ and let $y\in B(x,r)$. We
have that
\begin{equation*}
\begin{split}
\Big|\frac{f_n(x)-f_m(x)}{r}-\frac{f_n(y)-f_m(y)}{r}\Big|&\leq
\|\Lip (f_n-f_m)\|_{\infty}K_x\frac{d(x,y)}{r}\\&\leq \|\Lip
(f_n-f_m)\|_{\infty}K_x.
\end{split}
\end{equation*}

Let $\veps>0$. Since $\{f_n\}_n$ is a Cauchy sequence with respect
to the seminorm $\|\Lip (\cdot)\|_{\infty}$, there exists $n_1\geq
1$ such that if $n,m\geq n_1$, then
$$
\|\Lip (f_n-f_m)\|_{\infty}<\frac{\veps}{4K_x}.
$$

Thus, for each $r>0$  such that $B(x,r)\subset U_x$ and for each
$n,m\geq n_1$, we have the following chain of inequalities
\begin{multline*}
\Big|\frac{|f_n(x)-f_n(y)|}{r}-\frac{|f_m(x)-f_m(y)|}{r}\Big|\\\leq\Big|\frac{f_n(x)-f_m(x)}{r}-\frac{f_n(y)-f_m(y)}{r}\Big|\leq
\|\Lip (f_n-f_m)\|_{\infty}K_x<\frac{\veps}{4}
\end{multline*}
for each $y\in B(x,r)$.

In particular, for each $n\geq n_1$, we obtain that
\begin{multline*}
\frac{|f_n(x)-f_n(y)|}{r}\leq
\Big|\frac{|f_n(x)-f_n(y)|}{r}-\frac{|f_{n_1}(x)-f_{n_1}(y)|}{r}\Big|\\+\frac{|f_{n_1}(x)-f_{n_1}(y)|}{r}
<\frac{|f_{n_1}(x)-f_{n_1}(y)|}{r}+\frac{\veps}{4}.
\end{multline*}
Thus, the previous inequality implies, upon taking the supremum over
$B(x,r)$, that
$$
\sup_{y\in
B(x,r)}\Big\{\frac{|f_n(x)-f_n(y)|}{r}\Big\}\leq\sup_{y\in
B(x,r)}\Big\{\frac{|f_{n_1}(x)-f_{n_1}(y)|}{r}\Big\}+\frac{\veps}{4}
$$
for each $r>0$ such that $B(x,r)\subset U_x$.

On the other hand, for $\Lip(f_{n_1})(x)$, there exists $r_0>0$,
such that if $0<r<r_0$, then $B(x,r)\subset U_x$ and
$$
\sup_{y\in
B(x,r)}\Big\{\frac{|f_{n_1}(x)-f_{n_1}(y)|}{r}\Big\}\leq\Lip(f_{n_1})(x)+\frac{\veps}{4}.
$$
Hence, for each $n\geq n_1$ and each $0<r<r_0$, we obtain that
$$
\sup_{y\in
B(x,r)}\Big\{\frac{|f_n(x)-f_n(y)|}{r}\Big\}\leq\Lip(f_{n_1})(x)+\frac{2\veps}{4}.
$$

Since $f_n$ is a Cauchy sequence with respect to the seminorm
$\|\Lip (\cdot)\|_{\infty}$, then the sequence of real numbers
$\|\Lip (f_n)\|_{\infty}$ is a Cauchy sequence too and so there
exists $M>0$ such that $\|\Lip (f_n)\|_{\infty}<M$ for each
$n\geq1$. In particular, for each $n\geq n_1$ and $0<r<r_0$, we
obtain the following:
$$
\sup_{x\in
B(x,r)}\Big\{\frac{|f_n(x)-f_n(y)|}{r}\Big\}<\Lip(f_{n_1})(x)+\frac{2\veps}{4}\leq
\|\Lip (f_{n_1})\|_{\infty}+\frac{\veps}{2}\leq M+\frac{\veps}{2}.
$$

Now, let us see what happens with $f$. If $n\geq n_1$, $0<r<r_0$ and
$y\in B(x,r)$, we have that
\begin{multline*}
\frac{|f(x)-f(y)|}{r}\leq \frac{|f(x)-f_n(x)|}{r}+\frac{|f_n(x)-f_n(y)|}{r}+\frac{|f_n(y)-f(y)|}{r}\\
\leq
\frac{|f(x)-f_n(x)|}{r}+\frac{|f_n(y)-f(y)|}{r}+M+\frac{\veps}{2}.
\end{multline*}
Since $\{f_n\}_n$ converges uniformly to $f$, it converges pointwise
to $f$ and so there exists $n\geq n_1$ such that
$$
{|f(x)-f_n(x)|}+{|f_n(y)-f(y)|}<\frac{\veps r}{2}.
$$
Putting all above together we deduce that
$$
\frac{|f(x)-f(y)|}{r}\leq M+\veps.
$$
Thus, that inequality implies, upon taking the infimum over $B(x,r)$
and letting $r$ tending to $0$ that
$$ \Lip(f)(x)\leq M+\veps
$$
for each $x\in X$. Now, if $\veps\rightarrow 0$, we have that
$\Lip(f)(x)\leq M$ for each $x\in X$. And so $\|\Lip
f\|_{\infty}\leq M<+\infty$ which implies $f\in D(X)$.

To finish the proof, let us see that $\|\Lip
(f_n-f)\|_{\infty}\longrightarrow 0$. Using the above notation we
have that if $n,m\geq n_1$ and $0<r<r_0$
\begin{multline*}
\frac{|f_n(x)-f(x)-(f_n(y)-f(y))|}{r}\leq
\frac{|f_n(x)-f_m(x)-(f_n(y)-f_m(y))|}{r}\\+\frac{|f_m(x)-f(x)|}{r}+\frac{|f(y)-f_m(y)|}{r}
\leq
\frac{|f_m(x)-f(x)|}{r}+\frac{|f(y)-f_m(y)|}{r}+\frac{\veps}{4}.
\end{multline*}
The sequence $\{f_n\}_n$ converges uniformly to $f$ and, in
particular, it converges pointwise to $f$. Thus, there exists $n\geq
n_1$ such that
$$
{|f(x)-f_n(x)|}+{|f_n(y)-f(y)|}<\frac{\veps r}{2}.
$$
Hence, we have
$$
\frac{|f_n(x)-f(x)-(f_n(y)-f(y))|}{r}<\veps.
$$
Thus, we deduce that if $n\geq n_1$, then $\Lip (f_n-f)(x)\leq
\veps.$ This is true for each $x\in X$, and so we obtain that
$\|\Lip (f_n-f)\|_{\infty}\leq \veps$ if $n\geq n_1$. Therefore, we
have that $\|\Lip (f_n-f)\|_{\infty}\longrightarrow 0$. Thus, we
conclude that $(D^{\infty}(X),\|\cdot\|_{D^{\infty}})$ is a Banach
space as wanted.
\end{proof}

Let us see however that in general
$(D^{\infty}(X),\|\cdot\|_{D^{\infty}})$ is not a Banach space.
\begin{example}\em\label{nobanach}
Consider the connected metric space
$X=X_0\cup\bigcup_{n=1}^{\infty}X_n\cup G\subset\R^2 $ with the
metric induced by the Euclidean one, where
$X_0=\{0\}\times[0,+\infty)$, $X_n=\{\frac{1}{n}\}\times[0,n]$,
$n\in\N$ and $G=\{(x,\frac{1}{x}):0<x\leq 1\}$. For each $n\in\N$
consider the sequence of functions $f_n:X\rightarrow [0,1]$ given by
$$
f_n\Big(\frac{1}{k},y\Big)= \left\{
\begin{array}{ll}
\frac{k-y}{k\sqrt{k}}&\text{if $1\leq k\leq n$}\\[5pt]
0&\text{if $k>n$},
\end{array}
\right.
$$
and $f_n(x,y)=0$ if $x\neq \frac{1}{k}$ $\forall k\in\N$. Observe
that $f_n(\frac{1}{k},0)=\frac{1}{\sqrt{k}}$ and
$f_n(\frac{1}{k},k)=0$ if $1\leq k\leq n$. Since $\Lip
f_n(\frac{1}{k},y)=\frac{1}{k\sqrt{k}}$ and $\Lip f_n(x,y)=0$ if
$x\neq\frac{1}{k}$ $\forall k\in\N$, we have that $f_n\in
D^{\infty}(X)$ for each $n\geq1$. In addition, if $1<n<m$,
$$
\|f_n-f_m\|_{\infty}=\frac{1}{\sqrt{n+1}}\quad\text{ and }\quad
\|\Lip(f_n-f_m)\|_{\infty}=\frac{1}{(n+1)\sqrt{n+1}}.
$$
Thus, we deduce that $\{f_n\}_n$ is a Cauchy sequence in
$(D^{\infty}(X),\|\cdot\|_{D^{\infty}})$. However, if
$f_n\rightarrow f$ in $D^{\infty}$ then $f_n\rightarrow f$
pointwise. Then $f_m(\frac{1}{n},0)=\frac{1}{\sqrt{n}}$ for each
$m\geq n$ and so $f (\frac{1}{n},0)=\frac{1}{\sqrt{n}}$ and
$f(0,0)=0$. Thus, we obtain that
$$
\Lip(f)(0,0)\geq\lim_{n\rightarrow\infty}\frac{|f((\frac{1}{n}),0)-f(0,0)|}{d(\frac{1}{n},0)}=\lim_{n\rightarrow\infty}\frac{\frac{1}{\sqrt{n}}}{\frac{1}{n}}=+\infty,
$$ and so $f\notin D^{\infty}(X)$. This
means that $(D^{\infty}(X),\|\cdot\|_{D^{\infty}})$  is not a Banach
space.
\end{example}

\begin{thm}\label{DpuntoBanach}
Let $(X,d)$ be a connected locally radially quasi-convex metric
space and let $x_0\in X$. If we consider on $D(X)$ the norm
$\|\cdot\|_D=\max\{|f(x_0)|,\|\Lip (\cdot)\|_{\infty}\}$, then
$(D(X),\|\cdot\|_D)$ is a Banach space.
\end{thm}

\begin{proof}
By hypothesis, for each $y\in X$, there exists a neighborhood $U^y$
such that for each $z\in U^y$, there exists a rectifiable curve in
$U^y$ connecting $z$ and $y$. Since $X$ is connected, there exists a
finite sequence of points $y_1,\ldots,y_m$ such that $U^{y_k}\cap
U^{y_{k+1}}\neq \emptyset$ for $k=1,\ldots,m-1$, $x\in U^{y_1}$ and
$x_0\in U^{y_m}$. Now, for each $k=1\ldots m$, choose a point
$z_k\in U^{y_k}\cap U^{y_{k+1}}$.  To simplify notation we write
$z_0=x_0$ and $z_{n+1}=x$. For each $k=1\ldots m$, we choose a curve
$\gamma_k$ which connects $z_k$ with $z_{k+1}$. Taking
$\gamma=\gamma_{0}\cup\ldots\gamma_{m}$ we obtain a rectifiable
curve $\gamma$ which connects $x_0$ and $x$.

Let us see now that $(D(X),\|\cdot\|_D)$ is a Banach space. Indeed,
let $\{f_n\}_n$ be a Cauchy sequence. We consider the case on which
$f_n(x_0)=0$ for each $n\geq1$. The general case can be done in a
similar way. By combining the previous argument with Lemma
\ref{help}, we obtain that for $n,m\geq 1$ and for each $x\in X$, we
have that
$$
|f_n(x)-f_m(x)|\leq \|\Lip (f_n-f_m)\|_{\infty}\ell(\gamma)
$$
where $\gamma$ is a rectifiable curve connecting $x$ and $x_0$.
Since $\{f_n\}_n$ is a Cauchy sequence with respect to the seminorm
$\|\Lip (\cdot)\|_{\infty}$, the sequence $\{f_n(x)\}_n$ is a Cauchy
sequence  for each $x\in X$, and therefore, it converges to a point
$y=f(x)$. Then, in particular, $\{f_n\}_n$ converges pointwise to a
function $f:X\to\R$.

Next, one finds using the same strategy as in Theorem \ref{sol1}
(where we have just used the pointwise convergence) that a Cauchy
sequence $\{f_n\}_n\subset D(X)$ such that $f_n(x_0)=0$ for each
$n\geq 1$, converges in $(D(X),\|\cdot\|_D)$ to a function $f\in
D(X)$.
\end{proof}

We are now prepared to state the converse of Corollary \ref{main2}.
\begin{cor}\label{converse}
Let $(X,d)$ be a connected locally radially quasi-convex metric
space such that $\LIP(X)=D(X)$. Then $X$ is a quasi-length space.
\end{cor}
\begin{proof}
In view of Lemma \ref{casilon} we have to prove that there exists
$K>0$ such that for each $\varepsilon>0$ and each function
$f:X\longrightarrow\R$ we have that:
$$
|f(x)-f(y)|\leq K(d(x,y)+\varepsilon)\sup_{z\in
X}D_{\varepsilon}f(z)\qquad \forall x,y\in X\quad (*).
$$
Indeed, let $\varepsilon>0$. If $\sup_{z\in
X}D_{\varepsilon}f(z)=\infty$, then $(*)$ is trivially true. Thus,
we may assume that $\sup_{z\in X}D_{\varepsilon}f(z)<\infty$. Since
$\|\Lip f\|_{\infty}\leq \sup_{z\in X}D_{\varepsilon}f(z)$ then
$f\in D(X)$ and we distinguish two cases:
\begin{itemize}
\item[$(1)$] If $\|\Lip f\|_{\infty}=0$ then $f$ is locally constant and so constant because $X$ is connected. Therefore, the inequality trivially holds.
\item [$(2)$]If $\|\Lip f\|_{\infty}\neq0$, using that $f\in D(X)=\LIP(X)$, we have the
following inequality
$$
|f(x)-f(y)|\leq \LIP(f) d(x,y)\qquad\forall x,y\in X.
$$

Now, fix a point $x_0\in X$. Since $\LIP(X)=D(X)$ is a Banach space
with both norms
$$
\|f\|_{\LIP}=\max\{\LIP(f),|f(x_0)|\}\quad\text{ and
}\quad\|f\|_D=\max\{\|\Lip f\|_{\infty},|f(x_0)|\},
$$
(see Theorem \ref{DpuntoBanach} and e.g. \cite{Wea}) and
$\|\cdot\|_D\leq\|\cdot\|_{\LIP}$, then there exists a constant
$K>0$ such that $\|\cdot\|_{\LIP}\leq K\|\cdot\|_D$. Thus, if we
consider the function $g=f-f(x_0)$ we have that
$$
\LIP(f)=\LIP(g)=\|g\|_{\LIP}\leq K \|g\|_D=K \|\Lip
g\|_{\infty}=K\|\Lip f\|_{\infty}\quad(\heartsuit).
$$ Thus, we obtain that
\begin{equation*}
\begin{split}
|f(x)-f(y)|\leq &\LIP(f) d(x,y)\leq
\LIP(f)(d(x,y)+\varepsilon)\stackrel{(\heartsuit)}{\leq} K \|\Lip
f\|_{\infty}(d(x,y)+\varepsilon)\\\leq& K\sup_{z\in
X}D_{\varepsilon}f(z)(d(x,y)+\varepsilon)\quad\forall x,y\in X,
\end{split}
\end{equation*}
as wanted.
\end{itemize}
\end{proof}

\section{A Banach-Stone Theorem for infinitesimally Lipschitz functions}
There exist many results in the literature relating the topological
structure of a topological space $X$ with the algebraic or
topological-algebraic structures of certain function spaces defined
on it. The classical Banach-Stone theorem asserts that for a compact
space $X$, the linear metric structure of $\Cont(X)$ endowed with
the sup-norm determines the topology of $X$. Results along this line
for spaces of Lipschitz functions have been recently obtained in
\cite{GaJa2, GaJa3}.  In this section we prove two versions of the
Banach-Stone theorem for the function spaces $D^{\infty}(X)$ and
$D(X)$ respectively, where $X$ is a locally radially quasi-convex
space. Since in general $D(X)$ has not an algebra structure we will
consider on it its natural unital vector lattice structure. On the
other hand, on $D^{\infty}(X)$ we will consider both, its algebra
and its unital vector lattice structures.

The concept of real-valued infinitesimally Lipschitz function can be
generalized in a natural way when the target space is a metric
space.

\begin{defn}\em
Let $(X,d_X)$ and $(Y,d_Y)$ be metric spaces. Given a function
$f:X\rightarrow Y$ we define
$$
\Lip f(x)=\limsup_{\substack{y\to x\\ y\neq x}}
\frac{d_Y(f(x),f(y))}{d_X(x,y)}
$$
for each non-isolated $x\in X$. If $x$ is an isolated point we
define $\Lip f(x)=0$. We consider the following space of functions
$$
D(X,Y)=\{f:X\longrightarrow Y\,: \|\Lip f\|_{\infty}<+\infty\}.
$$
\end{defn}
As we have seen in Lemma \ref{dcont} we may observe that if $f\in
D(X,Y)$ then $f$ is  continuous. It can be easily checked
that we have also a Leibniz's rule in this context, that is, if
$f,g\in D^{\infty}(X)$, then $\|\Lip(f\cdot g)\|_{\infty}\leq \|\Lip
f\|_{\infty}\,\|g\|_{\infty}+ \|\Lip g\|_{\infty}\,\|f\|_{\infty}$.
In this way, we can always endow the space $D^{\infty}(X)$ with a natural
algebra structure. Note that $D^{\infty}(X)$ is {\it uniformly separating}
in the sense that for every pair of subsets $A$ and $B$ of $X$ with
$d(A,B)>0$, there exists some $f\in D^{\infty}(X)$ such that
$\overline{f(A)}\cap\overline{f(B)}=\emptyset$. In our case, if $A$
and $B$ are subsets of $X$ with $d(A,B)=\alpha>0$, then the function
$f=\inf\{d(\cdot,A),\alpha\}\in\LIP^{\infty}(X)\subset
D^{\infty}(X)$ satisfies that $f=0$ on $A$ and $f=\alpha$ on $B$. In
addition, we can endow either $D^{\infty}(X)$ or $D(X)$ with a
natural unital vector lattice structure.

We denote by ${\mathcal H}(D^{\infty}(X))$ the set of all nonzero
algebra homomorphisms $\varphi:D^{\infty}(X)\rightarrow \R$, that
is, the set of all nonzero multiplicative linear functionals on
$D^{\infty}(X)$. Note that in particular every algebra homomorphism
$\varphi \in {\mathcal H}(D^{\infty}(X))$ is positive, that is,
$\varphi(f)\geq 0$ when $f\geq 0$. Indeed, if $f$ and $1/f$ are in
$D^{\infty}(X)$, then $\varphi(f\cdot(1/f))=1$ implies that
$\varphi(f)\neq 0$ and $\varphi(1/f)=1/\varphi(f)$. Thus, if we
assume that $\varphi$ is not positive, then there exists $f\geq 0$
with $\varphi(f)<0$. The function $g=f-\varphi(f)\geq
-\varphi(f)>0$, satisfies $g\in D^{\infty}(X)$, $1/g\in
D^{\infty}(X)$ and $\varphi(g)=0$ which is a contradiction.

 Now, we endow ${\mathcal H}(D^{\infty}(X))$ with the
topology of pointwise convergence (that is, considered as a
topological subspace of $\R^{D^{\infty}(X)}$ with the product
topology). This construction is standard (see for instance
\cite{I}), but we give some details for completeness. It is easy to
check that ${\mathcal H}(D^{\infty}(X))$ is closed in
$\R^{D^{\infty}(X)}$ and therefore is a compact space. In addition,
since $D^{\infty}(X)$ separates points and closed sets, $X$ can be
embedded as a topological subspace of ${\mathcal H}(D^{\infty}(X))$
identifying each $x\in X$ with the point evaluation homomorphism
$\delta_x$ given by $\delta_x(f)=f(x)$, for every $f\in
D^{\infty}(X)$. We are going to see that $X$ is dense in ${\mathcal
H}(D^{\infty}(X))$. Indeed, given $\varphi\in {\mathcal
H}(D^{\infty}(X))$, $f_1,\ldots, f_n\in D^{\infty}(X)$, and
$\varepsilon>0$, there exists some $x\in X$ such that
$|\delta_x(f_i)-\varphi(f_i)|<\varepsilon$, for $i=1,\ldots,n$.
Otherwise, the function $g=\sum_{i=1}^n|f_i-\varphi(f_i)|\in
D^{\infty}(X)$ would satisfy $g\geq \varepsilon$ and $\varphi(g)=0$,
and this is impossible since $\varphi$ is positive. It follows that
${\mathcal H}(D^{\infty}(X))$ is a compactification of $X$.
Moreover, every $f\in D^{\infty}(X)$ admits a continuous extension
to ${\mathcal H}(D^{\infty}(X))$, namely by defining
$\widehat{f}(\varphi)=\varphi(f)$ for all $\varphi\in {\mathcal
H}(D^{\infty}(X))$.

\begin{lem}\label{homcontinuo}
Let $(X,d)$ be a metric space and $\varphi\in{\mathcal
H}(D^{\infty}(X))$. Then, $\varphi:D^{\infty}(X)\rightarrow \R$ is a
continuous map.
\end{lem}

\begin{proof}
Let $f\in D^{\infty}(X)$. We know that it admits a continuous
extension $\widehat{f}:{\mathcal H}(D^{\infty}(X))\rightarrow\R$ so
that $\widehat{f}(\varphi)=\varphi(f)$.
 Thus, since $X$ is dense in ${\mathcal H}(D^{\infty}(X))$,
 $$
 |\varphi(f)|=|\widehat{f}(\varphi)|\leq\sup_{\eta\in{\mathcal H}(D^{\infty}(X))}|\widehat{f}(\eta)|=\sup_{x\in
 X}|f(x)|\leq\|f\|_{D^{\infty}}
 $$
and we are done.
 \end{proof}

Recall that we have shown in Theorem \ref{sol1} that if $X$ is a
locally radially quasi-convex space then
$(D^{\infty}(X),\|\cdot\|_{D^{\infty}})$ is a Banach space. Using
this in a crucial way, we next give some results which will give
rise to a Banach-Stone theorem for $D^{\infty}(X)$.

\begin{lem}\label{grafcer}
Let $(X,d_X)$ and $(Y,d_Y)$ be locally radially quasi-convex metric
spaces. Then, every unital algebra homomorphism
$T:D^{\infty}(X)\rightarrow D^{\infty}(Y)$ is continuous for the
respective $D^{\infty}$-norms.
\end{lem}

\begin{proof}
In order to prove the continuity of the linear map $T$, we apply the
closed graph theorem. It is enough to check that given a sequence
$\{f_n\}_n\subset D^{\infty}(X)$ with $\|f_n-f\|_{ D^{\infty}}$
convergent to zero and $g\in D^{\infty}(X)$ such that
$\|T(f_n)-g\|_{ D^{\infty}}$ also convergent to zero, then $T(f)=g$.
Indeed, let $y\in Y$, and let $\delta_y\in {\mathcal
H}(D^{\infty}(Y))$ be the homomorphism given by the evaluation at
$y$, that is, $\delta_y(h)=h(y)$. By Lemma \ref{homcontinuo}, we
have that $\delta_y \circ T\in {\mathcal H}(D^{\infty}(X))$ is
continuous and so
$$
T(f_n)(y)=(\delta_y\circ T)(f_n)\to(\delta_y\circ T)(f)=T(f)(y)
$$
when $n\rightarrow \infty$.

On the other hand, since convergence in $D^{\infty}-$ norm implies
pointwise convergence, then $T(f_n)(y)$ converges to $g(y)$. That
is, $T(f)(y)=g(y)$, for each $y\in Y$. Hence, $T(f)=g$ as wanted.
\end{proof}

As a consequence, we obtain the following result concerning the
composition of infinitesimally Lipschitz functions.
\begin{prop}\label{hLip}
Let $(X,d_X)$ and $(Y,d_Y)$ be locally radially quasi-convex metric
spaces and let $h:X\rightarrow Y$. Suppose that $f\circ h\in
D^{\infty}(X)$ for each $f\in D^{\infty}(Y)$. Then $h\in D(X,Y)$.
\end{prop}

\begin{proof}
We begin by checking that $h$ is a continuous function, that is,
$h^{-1}(C)$ is closed in $X$ for each closed subset $C$ in $Y$. Let
$C$ be a closed subset of $Y$ and $x_0\in Y\backslash C$. Take
$f=\inf\{d(\cdot,C),d(x_0,C)\}\in D^{\infty}(Y)$ which satisfies
that $f(x_0)=1$ and $f(y)=0$ for each $y\in C$. Let us observe that
$f(x)=0$ if and only if $x\in C$ and so $f^{-1}(f(C))=C$. Thus,
since $f\circ h$ is continuous and $f(C)=0$ is closed in $\R$,
$h^{-1}(C)=h^{-1}(f^{-1}(f(C)))=(f\circ h)^{-1}(f(C))$ is closed in
$Y$.

\noindent Now, let $x_0\in X$ such that $h(x_0)$ is not an isolated
point. Note that if all points belonging to $h(X)$ are isolated we
have that $\Lip h(x)=0$ for each $x\in X$ and so $\|\Lip
h\|_{\infty}=0$, which implies that $h\in D(X,Y)$. Thus, we may
assume that there exists $x_0\in X$ such that $h(x_0)$ is not an
isolated point. Let $f_{x_0}=\min\{d_Y(\cdot,h(x_0)),1\}\in
D^{\infty}(Y)$, since it is a Lipschitz function. We have that
\begin{equation*}
\begin{split}
\Lip (f_{x_0}\circ h)(x_0)=&\limsup_{\substack{y\to x_0\\
y\neq x_0}} \frac{|f_{x_0}\circ h(y)-f_{x_0}\circ h(x_0)|}{d_X(x_0,y)}=\limsup_{\substack{y\to x_0\\
y\neq x_0}} \frac{|f_{x_0}\circ h(y)|}{d_X(x_0,y)}\\=&\limsup_{\substack{y\to x_0\\
y\neq x_0}}
\frac{|\min\{d_Y(h(y),h(x_0)),1\}|}{d_X(x_0,y)}\stackrel{(*)}{=}\Lip
h(x_0).
\end{split}
\end{equation*}
The equality $(*)$ holds because, as we have checked above, the map
$h$ is continuous. Thus, we obtain that
\begin{equation*}
\begin{split}
\Lip h(x_0)=&\Lip (f_{x_0}\circ h)(x_0)\leq \|\Lip (f_{x_0}\circ
h)\|_{\infty}\leq \|f_{x_0}\circ
h\|_{D^{\infty}(X)}\\\stackrel{(\heartsuit)}{\leq}&
K\|f_{x_0}\|_{D^{\infty}(Y)}\stackrel{(*)}{=}K\|\Lip
(f_{x_0})\|_{\infty}\quad (\dag)
\end{split}
\end{equation*}
 for a certain constant $K>0$ depending only on $g$. For
$(\heartsuit)$ we have used that, by Lemma \ref{grafcer}, the
homomorphism $T:D^{\infty}(Y)\rightarrow D^{\infty}(X)$,
$g\rightarrow g\circ h$ is continuous. The inequality $(*)$ holds
true because $\|f_x\|_{D^{\infty}(Y)}=\max\{\|f_{x_0}\|_{\infty},
\|\Lip (f_{x_0})\|_{\infty}\}$, $\|f_{x_0}\|_{\infty}\leq 1$ and
$\|\Lip (f_{x_0})\|_{\infty}=1$. It remains to check that, $\|\Lip
(f_{x_0})\|_{\infty}=1$. Indeed,
$$
\Lip f_{x_0}(z)=\limsup_{\substack{z'\to z\\
z'\neq z}} \frac{|f_{x_0}(z')-f_{x_0}(z)|}{d_Y(z',z)}.
$$
We have to distinguish three different cases:
\begin{itemize}
\item[(i)] If $d_Y(z,h(x_0))>1$, there exists a neighborhood $V_z$
where\\ $d_Y(z',h(x_0))>1$ for each $z'\in V_z$ and
${f_{x_0}}_{|_{V_z}}=1$. Thus, $\Lip f_{x_0}(z)=0$.
\item[(ii)] If $d_Y(z,h(x_0))<1$, there exists a neighborhood
$V_z$ where \\$d_Y(z',h(x_0))>1$ for each $z'\in V_z$ and so
$$
\Lip f_{x_0}(z)=\limsup_{\substack{z'\to z\\
z'\neq z}}\frac{|d_Y(z',h(x_0))-d_Y(z,h(x_0))|}{d_Y(z',z)}\leq \limsup_{\substack{z'\to z\\
z'\neq z}}\frac{d_Y(z',z)}{d_Y(z',z)}=1.
$$
\item[(iii)] If $d_Y(z,h(x_0))=1$, then
$$
\Lip f_{x_0}(z)=\limsup_{\substack{z'\to z\\
z'\neq z}}\frac{1-\min\{d_Y(z',h(x_0)),1\}}{d_Y(z',z)}.
$$
If $d_Y(z',h(x_0))\geq 1$, then $1-\min\{d_Y(z',h(x_0)),1\}=0$. On
the other hand, if $d_Y(z',h(x_0))< 1$, then
$$1-\min\{d_Y(z',h(x_0)),1\}=d_Y(z,h(x_0))-d_Y(z',h(x_0))\leq
d_Y(z,z').$$ Hence, we deduce that $\Lip f_{x_0}(z)\leq 1$.
\end{itemize}
On the other hand, since $h(x_0)$ is not an isolated point, $\Lip
f_{x_0}(z)=1$ and so $\|\Lip f\|_{\infty}=1$ because we have seen
that $\|\Lip f\|_{\infty}\leq 1$. Then, upon taking the supremum
over $X$ in both sides of the inequality $(\dag)$ we conclude that
$\|\Lip h\|_{\infty}\leq K$, as wanted.
\end{proof}

\begin{remark}\em
If we look at Theorem $3.12$ in \cite{GaJa2}, where  an analogous
result to Proposition \ref{hLip} for Lipschitz functions is
obtained, we can see that the  argument there is based on the fact
that the distance can be expressed in terms of Lipschitz functions.
In our case, we cannot use the same strategy since we do not know
how to compare the values of an infinitesimally Lipschitz function
at two arbitrary points of the space.
\end{remark}

Finally, we need the following useful Lemma, which shows that the
points in $X$ can be topologically distinguished into ${\mathcal
H}(D^{\infty}(X))$. It is essentially known (see for instance
\cite{GaJa1}) but we give a proof for completeness.

\begin{lem}\label{superlema}
Let $(X,d)$ be a complete metric space and let $\varphi\in{\mathcal
H}(D^{\infty}(Y))$. Then $\varphi$ has a countable neighborhood
basis in ${\mathcal H}(D^{\infty}(X))$ if, and only if, $\varphi\in
X$.
\end{lem}

\begin{proof}
Suppose first that $\varphi\in{\mathcal H}(D^{\infty}(Y))\backslash
X$ has a countable neighborhood basis. Since $X$ is dense in
${\mathcal H}(D^{\infty}(X))$, there exists a sequence $(x_n)$ in
$X$ converging to $\varphi$.
 The completeness of $X$ implies that $(x_n)$ has no $d-$Cauchy sequence, and therefore there exist
 $\varepsilon>0$ and a subsequence $(x_{n_k})$ such that
 $d(x_{n_k},x_{n_j})\geq\varepsilon$ for $k\neq j$. Now, the sets $A=\{x_{n_k}: k\text{ even
 }\}$ and $B=\{x_{n_k}: k\text{ odd }\}$ satisfy
 $d(A,B)\geq\varepsilon$, and since $D^{\infty}(X)$ is uniformly separating, there is a function $f\in
 D^{\infty}(X)$ with $\overline{f(A)}\cap\overline{f(B)}=\emptyset$.
 But this is a contradiction since $f$ extends continuously to ${\mathcal
 H}(D^{\infty}(X))$ and $\varphi$ is in the closure of both $A$ and
 $B$.

 Conversely, if $\varphi\in X$, consider $B_n$ the open ball in $X$
 with center $\varphi$ and radius $1/n$. Then the family $\{\overline{B_n}\}_n$ of the closures of $B_n$ in ${\mathcal
 H}(D^{\infty}(X))$ is easily seen to be a countable
 neighborhood basis as required.
 \end{proof}

Now, we are in a position to show that the algebra structure of
$D^{\infty}(X)$ determines the infinitesimal Lipschitz structure of
a complete locally radially quasi-convex metric space. We say that
two metric spaces $X$ and $Y$ are \em infinitesimally Lipschitz
homeomorphic \em if there exists a bijection $h:X\rightarrow Y$ such
that $h\in D(X,Y)$ and $h^{-1}\in D(Y,X)$.
\begin{thm}\textbf{\em(Banach-Stone type)\em}\label{banachstone}
Let $(X,d_X)$ and $(Y,d_Y)$ be complete locally radially
quasi-convex metric spaces. The following are equivalent:
\begin{itemize}
\item[(a)]$X$ is  infinitesimally Lipschitz homeomorphic to $Y$.
\item[(b)]$D^{\infty}(X)$ is isomorphic to $D^{\infty}(Y)$ as unital algebras.
\item[(c)]$D^{\infty}(X)$ is isomorphic to $D^{\infty}(Y)$ as unital vector lattices.
\end{itemize}
\end{thm}

\begin{proof}
$(a)\Longrightarrow(b)$  If $h:X\rightarrow Y$ is an infinitesimally
Lipschitz homeomorphism, then it is easy to check the map
$T:D^{\infty}(Y)\rightarrow D^{\infty}(X)$, $f\mapsto T(f)=f\circ
h$, is an isomorphism of unital algebras.

$(b)\Longrightarrow(a)$ Let $T:D^{\infty}(X)\rightarrow
D^{\infty}(Y)$ be an isomorphism of unital algebras. We define
$h:{\mathcal H}(D^{\infty}(Y))\rightarrow{\mathcal
H}(D^{\infty}(X))$, $\varphi\mapsto h(\varphi)=\varphi\circ T$. Let
us see first that $h$ is an homeomorphism. To reach that aim, it is
enough to prove that $h$ is bijective, closed and continuous. Since
$T$ is an isomorphism, $h^{-1}(\psi)=\psi\circ T^{-1}$ exists for
every $\psi\in {\mathcal H}(D^{\infty}(X))$, and so $h$ is
bijective. In addition, once we check that $h$ is continuous we will
also have that $h$ is closed because ${\mathcal H}(D^{\infty}(Y))$
is compact and ${\mathcal H}(D^{\infty}(X))$ is a Hausdorff space.
Now consider the following diagram:
$$
\xymatrix{Y\ar[d]_{T(f)}\ar@{^{(}->}[r]& \ar[r]^{h}\ar[d]_{\widehat{T(f)}}\ar@{-->}[rd]_{\widehat{f}\circ h}{\mathcal H}(D^{\infty}(Y))&\ar[d]^{\widehat{f}}{\mathcal H}(D^{\infty}(X))\ar@{<-^{)}}[r]&X\ar[d]^f\\
\R&\R&\R&\R}
$$

Here, $\widehat{f}$ (respectively $\widehat{T(f)}$) denotes the
continuous extension of $f$ (respectively $T(f)$) to ${\mathcal
H}(D^{\infty}(X))$. Thus, $h$ is continuous if and only if
$\widehat{f}\circ h$ is continuous for all $f\in D^{\infty}(X)$.
Hence, it is enough to prove that $\widehat{f}\circ
h=\widehat{T(f)}$. Since $X$ is dense in ${\mathcal
H}(D^{\infty}(X))$, it is suffices to check that
$$
\widehat{T(f)}(\delta_x)=\widehat{f}\circ h(\delta_x),
$$
where $\delta_x$ denotes the evaluation homomorphism for each $x\in
X$. It is clear that,
$$
\widehat{f}\circ h(\delta_x)=(h\circ\delta_x)(f)=(\delta_x\circ
T)(f)=\delta_{T(f)(x)}=\delta_x(Tf)=\widehat{T(f)}(\delta_x),
$$
and so $h$ is continuous.

By Lemma \ref{superlema} we have that a point $\varphi\in{\mathcal
H}(D^{\infty}(X))$ has a countable neighborhood basis in ${\mathcal
H}(D^{\infty}(X))$ if and only if it corresponds to a point of $X$.
Since the same holds for $Y$ and ${\mathcal H}(D^{\infty}(Y))$ we
conclude that $h(Y)=X$ and by Proposition \ref{hLip} we have that
$h|_{Y}\in D(Y,X)$. Analogously, $h^{-1}|_{X}\in D(X,Y)$ and so $X$
and $Y$ are infinitesimally Lipschitz homeomorphic.

To prove $(b)\iff(c)$ We use that $D^{\infty}(X)$ is \em closed
under bounded inversion \em which means that if $f\in D^{\infty}(X)$
and $f\geq 1$, then $1/f\in D^{\infty}(X)$. Indeed, if $f\in
D^{\infty}(X)$ and $f\geq 1$, given $\varepsilon>0$ there exists
$r>0$ such that
$$
\frac{|f(x)-f(y)|}{d(x,y)}\leq\sup_{\substack{d(x,y)\leq r\\
y\neq x}}\frac{|f(x)-f(y)|}{d(x,y)}\leq M+\varepsilon\qquad\forall
y\in B(x,r)\quad(\star).
$$
Thus, given $x\in X$,
$$
\Big|\frac{1}{f(y)}-\frac{1}{f(x)}\Big|=\frac{|f(x)-f(y)|}{|f(x)f(y)|}\stackrel{(*)}{\leq}d(x,y)(M+\varepsilon)\qquad\forall
y\in B(x,r),
$$
where inequality $(*)$ is obtained after applying $(\star)$ and the
fact that $|f(x)f(y)|\geq 1$. Thus, the conclusion follows from
Lemma $2.3$ in \cite{GaJa2}.
\end{proof}

\begin{cor}\label{banachstone2}
Let $(X,d_X)$ and $(Y,d_Y)$ be complete locally radially
quasi-convex metric spaces. The following assertions are equivalent:
\begin{itemize}
\item[(a)] $X$ is  infinitesimally Lipschitz homeomorphic to $Y$.
\item[(b)] $D(X)$ is isomorphic to $D(Y)$ as unital vector lattices.
\end{itemize}
\end{cor}
\begin{proof}
$(a)\Longrightarrow(b)$ If $h:X\rightarrow Y$ is an infinitesimally
Lipschitz homeomorphism, then it is clear that the map
$T:D(Y)\rightarrow D(X)$, $f\mapsto T(f)=f\circ h$, is an
isomorphism of unital vector lattices.

\vspace{2mm} $(b)\Longrightarrow(a)$ It follows from Theorem
\ref{banachstone}, since each homomorphism of unital vector lattices
$T:D(Y)\rightarrow D(X)$ takes bounded functions to bounded
functions. Indeed, if $|f|\leq M$ then $ |T(f)|=T(|f|)\leq T(M)=M.$
\end{proof}

Next we deal with  what we call \em infinitesimal isometries \em
between metric spaces, related to infinitesimally Lipschitz
functions.
\begin{defn}\em
Let $(X,d_X)$ and $(Y,d_Y)$ be metric spaces. We say that $X$ and
$Y$ are \em infinitesimally  isometric \em if there exists a
bijection $h:X\longrightarrow Y$ such that $\|\Lip
h\|_{\infty}=\|\Lip h^{-1}\|_{\infty}=1$.
\end{defn}

\begin{remark}\em
We deduce from the proofs of Proposition \ref{hLip} and Theorem
\ref{banachstone} that two complete locally radially quasi-convex
metric spaces $X$ and $Y$ are infinitesimally  isometric if, and
only if, there exists an algebra isomorphism $T:D^{\infty}(Y)\to
D^{\infty}(X)$ which is an isometry for the
$\|\cdot\|_{D^{\infty}}$-norms (that is, $\|T\|=\|T^{-1}\|=1$).
\end{remark}

It is clear that if two metric spaces are locally isometric, then
they are infini\-te\-simally  isometric. The converse is not true,
as we can see  throughout the following example.

\begin{example}\em
Let $(X,d)$ be the metric space introduced in Example \ref{cuspide}
and let $(Y,d')$ be the metric space defined in the following way.
Consider the interval $Y=[-1,1]$ and let us define a metric on it as
follows:
$$
d'(t,s)=\left\{\begin {array}{ll} d((t^3,t^2),(s^3,s^2))&\text{ if } t,s\in[-1,0],\\[10pt]
d((t^3,t^2),(s^3,s^2))&\text{ if }
t,s\in[0,1],\\[10pt]d((t^3,t^2),(0,0))+d((0,0),(s^3,s^2))&\text{ if }
t\in[-1,0], s\in[0,1].
\end{array}
\right.
$$
It is easy to see that $d'$ defines a metric. We define
$$
h:X\to Y,\ (t^3,t^2)\to t.
$$
Let us observe that $\|\Lip h\|_{\infty}=\|\Lip h^{-1}\|_{\infty}=1$
and so $X$ and $Y$ are infinitesimally  isometric. However, at the
origin $(0,0)$, for each $r>0$ we have that
$$
d(z,y)\neq d'(h(z),h(y))\quad\forall z,y\in B((0,0),r).
$$
Thus, $h$ is an infinitesimal  isometry, but not a local isometry.
In fact, it can be checked that there is no local isometry
$f:X\longrightarrow Y$.
\end{example}

\begin{steps}{Non complete case.} If $X$ is a metric space and
$\widetilde{X}$ denotes its completion, then both metric spaces have
the same uniformly continuous functions. Therefore,
$\LIP(X)=\LIP(\widetilde{X})$, and completeness of spaces cannot be
avoided in the Lipschitzian case. We are interested in how
completeness assumption works for the $D$-case. It would be useful
to analyze if there exists a Banach-Stone theorem for not complete
metric spaces.
\end{steps}
\begin{example}\em
Let $(X,d)$ be the metric space given by
$$ X=\{(x,y)\in\R^2:y^3=x^2,
-1\leq x\leq 1\}=\{(t^3,t^2),-1\leq t\leq 1\},
$$
where $d$ is the restriction to $X$ of the Euclidean metric of
$\R^2$. Let $(Y,d')$ be the metric space given by
$Y=X\backslash\{0\}$ and $d'=d|_Y$. Observe that $(X,d)$ is the
completion of $(Y,d')$. The function
$$
h:Y\to \R,\ (x,y)\mapsto\left\{\begin{array}{cl}1&\text{ if }x< 0\\
0&\text{ if }x>0,\end{array}\right.
$$
belongs to $D(Y)$ but $h$ cannot be even continuously extended to
$X$. Thus, $D(Y)\neq D(X)$.
\end{example}

In the following example we construct a metric space $X$ such that
$D(X)=D(\widetilde{X})$, where $\widetilde{X}$ denotes the
completion of $X$, and so that $X$ is not homeomorphic to
$\widetilde{X}$. This fact illustrates that, a priori, one cannot
expect a conclusive result for the non complete case.

\begin{example}\em
Let $X$ be a metric space defined as follows:
$$
X=\{(t^3,t^2),-1\leq t\leq 1\}\cup\{(x,1)\in\R^2\,: 1\leq x<
2\}=A\cup B.
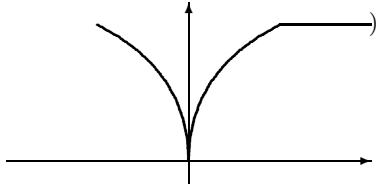
$$
\begin{center}\setlength{\unitlength}{0.6mm}
\begin{picture}(100,40)
\put(0,0){\vector(1,0){80}} \put(40,-5){\vector(0,1){40}}
\thicklines \qbezier(20,30)(40,20)(40,0)
\qbezier(60,30)(40,20)(40,0) \drawline(60,30)(80,30)
 \put(79.5,28.8){\small$\displaystyle)$}
\end{picture}
\end{center}

\vspace{2mm} Now, we consider the completion of $X$ :
$$
\widetilde{X}=\{(t^3,t^2),-1\leq t\leq 1\}\cup\{(x,1)\in\R^2:\,
1\leq x\leq 2\}=\widetilde{A}\cup\widetilde{B}.
$$
Let $f\in D(X)$. First of all, $D(B)=\LIP(B)$, since $B$ is a
quasi-length space, and so, by McShane's theorem (see \cite{hei}),
there exists $F\in \LIP(\widetilde{B})$ such that $F|_B=f$. Thus,
$$
G(x,y)=\left\{
\begin{array}{ll}
f(x,y)&\text{if $(x,y)\in A=\widetilde{A}$}\\[4pt]
F(x,y)&\text{if $(x,y)\in \widetilde{B}$},
\end{array}
\right.
$$
is a $D-$extension of $f$ to the completion $\widetilde{X}$. And so
$D(X)=D(\widetilde{X})$. However, $X$ is not homeomorphic to
$\widetilde{X}$ since $\widetilde{X}$ is compact but $X$ is not.
\end{example}

\section{Sobolev spaces on metric measure spaces}
Along this section, we always assume that $(X,d,\mu)$ is a metric
measure space, where $\mu$ is a \em Borel regular measure \em, that
is, $\mu$ is an outer measure on a metric space $(X, d)$ such that
all Borel sets are $\mu-$measurable and for each set $A\subset X$
there exists a Borel set $B$ such that $A\subset B$ and
$\mu(A)=\mu(B)$.

Our aim in this section is to compare the function spaces
$D^{\infty}(X)$ and $\LIP^{\infty}(X)$ with certain Sobolev spaces
on metric-measure spaces. There are several possible extensions of
the classical theory of Sobolev spaces to the setting of metric
spaces equipped with a Borel measure. Following \cite{Amb} and
\cite{haj} we record the definition of $M^{1,p}$ spaces:

\begin{steps}{Haj{\l}asz-Sobolev space.}
For $0<p\leq\infty$ the space $\widetilde{M}^{1,p}(X,d,\mu)$ is
defined as the set of all functions $f\in L^p(X)$ for which there
exists a function $0\leq g\in L^p(X)$ such that
$$
|f(x)-f(y)|\leq d(x,y)(g(x)+g(y))\qquad\mu-a.e. \quad(*).
$$
As usual, we get the space $M^{1,p}(X,d,\mu)$ after identifying any
two functions $u,v\in \widetilde{M}^{1,p}(X,d,\mu)$ such that $u=v$
almost everywhere with respect to $\mu$. The space
$M^{1,p}(X,d,\mu)$ is equipped with the norm
$$
\|f\|_{M^{1,p}}=\|f\|_{L^p}+\inf_{g}\|g\|_{L^p},
$$
where the infimum is taken over all functions $0\leq g\in L^p(X)$
that satisfy the requirement $(*)$.

In particular, if $p=\infty$ it can be shown that
$M^{1,\infty}(X,d,\mu)$ coincides with $\LIP^{\infty}(X)$ provided
that $\mu(B)>0$ for every open ball $B\subset X$ (see \cite{Amb})
and that $1/2\|\cdot\|_{\LIP^{\infty}}\leq
\|\cdot\|_{M^{1,\infty}}\leq \|\cdot\|_{\LIP^{\infty}}$. In this
case we obtain that $M^{1,\infty}(X)=\LIP^{\infty}(X)\subseteq
D^{\infty}(X)$.
\end{steps}

\begin{steps}{Newtonian space.}
Another interesting generalization of Sobolev spaces to general
metric spaces are the so-called Newtonian Spaces, introduced by
Shanmungalingam \cite{Sha,Sh1}. Its definition is based on the
notion of the upper gradient that we recall here for the sake of
completeness.

A non-negative Borel function $g$ on $X$ is said to be an \em upper
gradient \em for an extended real-valued function $f$ on X, if
$$
|f(\gamma(a))-f(\gamma(b))|\leq\int_{\gamma}g\qquad(*)
$$
for every rectifiable curve $\gamma:[a,b]\rightarrow X$. We see that
the upper gradient plays the role of a derivative in the formula
$(*)$ which is similar to the one related to the fundamental theorem
of calculus. The point is that using upper gradients we may have
many of the properties of ordinary Sobolev spaces even though we do
not have derivatives of our functions.

If $g$ is an upper gradient of $u$ and $\widetilde{g}=g$ almost
everywhere, then it may happen that $\widetilde{g}$ is no longer an
upper gradient for $u$. We do not want our upper gradients to be
sensitive to changes on small sets. To avoid this unpleasant
situation the notion of \em weak upper gradient \em is introduced as
follows. First we need a way to measure how large a family of curves
is. The most important point is if a family of curves is small
enough to be ignored. This kind of problem was first approached in
\cite{Fu}. In what follows let $\Upsilon\equiv\Upsilon(X)$ denote
the family of all nonconstant rectifiable curves in $X$. It may
happen $\Upsilon=\emptyset$, but we will be mainly concerned with
metric spaces for which the space $\Upsilon$ is large enough.

\begin{defn}\em (Modulus of a family of curves)
Let $\Gamma\subset\Upsilon$. For $1\leq p<\infty$ we define the
$p-$\em modulus of $\Gamma$ \em by
$$
\Mod_{p}(\Gamma)=\inf_{\rho}\int_{X}\rho^p\,d\mu,
$$
where the infimum is taken over all non-negative Borel functions
$\rho:X\rightarrow[0,\infty]$ such that $\int_{\gamma}\rho\geq 1$
for all $\gamma\in\Gamma$. If some property holds for all curves
$\gamma\in\Upsilon\backslash\Gamma$, such that $\Mod_p{\Gamma}=0$,
then we say that the property holds for \em $p-$a.e. curve\em.
\end{defn}

\begin{defn}\em
A non-negative Borel function $g$ on $X$ is a \em $p-$weak upper
gradient \em of an extended real-valued function $f$ on X, if
$$
|f(\gamma(a))-f(\gamma(b))|\leq\int_{\gamma}g
$$
for $p-$a.e. curve $\gamma\in\Upsilon$.
\end{defn}
Let $\widetilde{N}^{1,p}(X,d,\mu)$, where $1\leq p<\infty$, be the
class of all $L^p$ integrable Borel functions on $X$ for which there
exists a $p-$weak upper gradient in $L^p$. For
$f\in\widetilde{N}^{1,p}(X,d,\mu)$ we define
$$
\|u\|_{\widetilde{N}^{1,p}}=\|u\|_{L^p}+\inf_{g}\|g\|_{L^p},
$$
where the infimum is taken over all $p-$weak upper gradients $g$ of
$u$. Now, we define in $\widetilde{N}^{1,p}$  an equivalence
relation  by $u\sim v$ if and only if
$\|u-v\|_{\widetilde{N}^{1,p}}=0$. Then the space $N^{1,p}(X,d,\mu)$
is defined as the quotient $\widetilde{N}^{1,p}(X,d,\mu)/\sim$ and
it is equipped with the norm
$\|u\|_{N^{1,p}}=\|u\|_{\widetilde{N}^{1,p}}.$

Next, we consider the case $p=\infty$. We will introduce the
corresponding definition of $\infty-$modulus of a family of
rectifiable curves which will be an important ingredient for the
definition of the Sobolev space $N^{1,\infty}(X)$.

\begin{defn}\em
For $\Gamma\subset\Upsilon$, let $F(\Gamma)$ be the family of all
Borel measurable functions $\rho:X\rightarrow[0,\infty]$ such that
$$
\int_{\gamma}\rho\geq 1\,\,\text{ for all
 }\,\gamma\in\Gamma.
$$
We define the $\infty-$\em modulus of $\Gamma$ \em by
$$
\Mod_{\infty}(\Gamma)=\inf_{\rho\in
F(\Gamma)}\{\|\rho\|_{L^{\infty}}\}\in[0,\infty].
$$
If some property holds for all curves
$\gamma\in\Upsilon\backslash\Gamma$, where
$\Mod_{\infty}{\Gamma}=0$, then we say that the property holds for
\em $\infty-$a.e. curve\em.
\end{defn}
\begin{remark}\em
It can be easily checked that $\Mod_{\infty}$ is an outer measure as
it happens for $1\leq p<\infty$. See for example Theorem $5.2$ in
\cite{haj}.
\end{remark}

Next, we provide a characterization of path families whose
$\infty-$modulus is zero.

\begin{lem}\label{modinfzero}
Let $\Gamma\subset\Upsilon$. The following conditions are
equivalent:
\begin{itemize}
\item[(a)] $\Mod_{\infty}{\Gamma}=0$.
\item[(b)] There exists a Borel  function $0\leq\rho\in L^{\infty}(X)$ such that
$\int_{\gamma}\rho=+\infty$, for each $\gamma\in\Gamma$.
\item[(c)]There exists a Borel  function $0\leq\rho\in L^{\infty}(X)$ such that
$\int_{\gamma}\rho=+\infty$, for each $\gamma\in\Gamma$ and
$\|\rho\|_{L^{\infty}}=0$.
\end{itemize}
 \end{lem}

\begin{proof}
$(a)\Rightarrow(b)$ If $\Mod_{\infty}{\Gamma}=0$, for each $n\in\N$
there exists $\rho_n\in F(\Gamma)$ such that
$\|\rho_n\|_{L^{\infty}}<1/2^n$. Let $\rho=\sum_{n\geq 1}\rho_n$.
Then $\|\rho\|_{L^{\infty}}\leq\sum_{n=1}^{\infty}1/2^n=1$ and
$\int_{\gamma}\rho=\int_{\gamma}\sum_{n\geq 1}\rho_n=\infty$.

$(b)\Rightarrow(a)$ On the other hand, let $\rho_n=\rho/n$ for all
$n\in\N$. By hypothesis $\int_{\gamma}\rho_n=\infty$ for all
$n\in\N$
 and $\gamma\in\Gamma$. Then $\rho_n\in F(\Gamma)$ and $\|\rho\|_{L^{\infty}}/n\rightarrow 0$ as $n\rightarrow\infty$. Hence
 $\Mod_{\infty}(\Gamma)=0$.

 $(b)\Rightarrow(c)$ By hypothesis there exists a Borel
measurable function $0\leq\rho\in L^{\infty}(X)$ such that,
$$
\int_{\gamma}\rho=+\infty\,\,\text{ for every }\gamma\in\Gamma.
$$
Consider the function
$$
h(x)= \left\{
\begin{array}{ll}
\|\rho\|_{L^{\infty}}&\text{if $\|\rho\|_{L^{\infty}}\geq \rho(x),$}\\[5pt]
\infty&\text{if $\rho(x)>\|\rho\|_{L^{\infty}}$.}
\end{array}
\right.
$$
Notice that $\|\rho\|_{L^{\infty}}=\|h\|_{L^{\infty}}$, and since
$\int_{\gamma}\rho=+\infty$ for every $\gamma\in\Gamma_1$ and
$\rho\leq h$, we have that $\int_{\gamma}h=+\infty$ for every
$\gamma\in\Gamma_1$. Now, we define the function
$\varrho=h-\|h\|_{L^{\infty}}$ which has
$\|\varrho\|_{L^{\infty}}=0$ and
$$
\int_{\gamma}\varrho=\int_{\gamma}h-\|h\|_{L^{\infty}}\ell(\gamma)=+\infty\,\,\text{
for every }\gamma\in\Gamma_1.
$$
\end{proof}

Now we are ready to define the notion of \em $\infty-$weak upper
gradient \em.
\begin{defn}\em
A non-negative Borel function $g$ on $X$ is an \em $\infty-$weak
upper gradient \em of an extended real-valued function $f$ on X, if
$$
|f(\gamma(a))-f(\gamma(b))|\leq\int_{\gamma}g
$$
for $\infty-$a.e. curve every curve $\gamma\in\Upsilon$.
\end{defn}
Let $\widetilde{N}^{1,\infty}(X,d,\mu)$, be the class of all
functions $f\in L^\infty(X)$ Borel for which there exists an
$\infty-$weak upper gradient in $L^\infty$. For
$f\in\widetilde{N}^{1,\infty}(X,d,\mu)$ we define
$$
\|u\|_{\widetilde{N}^{1,\infty}}=\|u\|_{L^\infty}+\inf_{g}\|g\|_{L^\infty},
$$
where the infimum is taken over all $\infty-$weak upper gradients
$g$ of $u$.

\begin{defn}\em(Newtonian space for $p=\infty$)
We define an equivalence relation in $\widetilde{N}^{1,\infty}$ by
$u\sim v$ if and only if $\|u-v\|_{\widetilde{N}^{1,\infty}}=0$.
Then the space $N^{1,\infty}(X,d,\mu)$ is defined as the quotient
$\widetilde{N}^{1,\infty}(X,d,\mu)/\sim$ and it is equipped with the
norm
$$
\|u\|_{N^{1,\infty}}=\|u\|_{\widetilde{N}^{1,\infty}}.
$$
\end{defn}

Note that if $u\in \widetilde{N}^{1,\infty}$ and $v=u$ $\mu-$a.e.,
then it is not necessarily true that $v\in
\widetilde{N}^{1,\infty}$. Indeed, let $(X=[-1,1],d,\lambda)$ where
$d$ denotes the Euclidean distance and $\lambda$ the Lebesgue
measure. Let $u:X\rightarrow \R$ be the function $u=1$ and
$v:X\rightarrow\R$ given by $v=1$ if $x\neq 0$ and $v(x)=\infty$ if
$x=0$. In this case we have that $u=v$ $\mu-a.e.$, $u\in
\widetilde{N}^{1,\infty}$ but $v\notin \widetilde{N}^{1,\infty}$. It
can be shown that if $u,v\in \widetilde{N}^{1,\infty}$, and  $v=u$
$\mu-$a.e., then $\|u-v\|_{\widetilde{N}^{1,\infty}}=0$. In
addition, $N^{1,\infty}(X)$ is a Banach space. Both results can be
checked adapting properly the respective proofs for the case
$p<\infty$. For further details see \cite{Sh1}.

\begin{lem}\label{lipgradient}
If $f\in D(X)$ then $\Lip (f)$ is an upper gradient of $f$.
\end{lem}

\begin{proof}
 Let $\gamma:[a,b]\rightarrow X$ be a rectifiable curve parametrized by arc-length which connects $x$ and $y$.
  It can be checked that $\gamma$ is $1-$Lipschitz (see for instance Theorem $3.2$ in \cite{haj}). The function $f\circ\gamma$
  is an infinitesimally Lipschitz function and by Stepanov's differentiability theorem (see \cite{BRZ}), it
is differentiable a.e. Note that $|(f\circ\gamma)'(t)|\leq\Lip
f(\gamma(t))$ at every point of $[a,b]$  where $(f\circ\gamma)$ is
differentiable. Now, we deduce that
 $$
 |f(x)-f(y)|\leq\Big|\int_a^b(f\circ\gamma)^{'}(t)dt\Big|\leq\int_a^b\Lip(f(\gamma(t)))\,dt
 $$
 as wanted.
\end{proof}

Now suppose that $\mu(B)>0$ for every open ball $B\subset X$. It is
clear by Lemma \ref{lipgradient} that $D^{\infty}(X)\,\subset
\widetilde{N}^{1,\infty}(X)$ and that the map
$$
\begin{array}{ccccl}
\phi&:&D^{\infty}(X)&\longrightarrow& N^{1,\infty}(X)\\
&&f&\longrightarrow&[f].
\end{array}
$$
is an inclusion. Indeed, if $f,g\in D^{\infty}(X) $ with $0=[f-g]\in
N^{1,\infty}(X)$, we have  $f-g=0$ $\mu-$a.e. Thus $f=g$ in a dense
subset and since $f,g$ are continuous we obtain that $f=g$.
Therefore we have the following chain of inclusions:
$$
\LIP^{\infty}(X)=M^{1,\infty}(X)\subset D^{\infty}(X)\subset
N^{1,\infty}(X),\qquad (*)
$$
and $\|\cdot\|_{N^{1,\infty}}\leq \|\cdot\|_{D^{\infty}}\leq
\|\cdot\|_{\LIP^{\infty}}\leq 2\,\|\cdot\|_{M^{1,\infty}}$.
\end{steps}The next example shows that in general $D^{\infty}(X)\neq
N^{1,\infty}(X)$.
\begin{example}\em
Consider the metric space $(X=\{B_n\}_n,d_e)$, where $d_e$ is the
restriction to $X$ of the Euclidean metric of $\R^2$ and $\{B_n\}_n$
is a sequence of open balls with radius convergent to zero, as shows
the picture below:
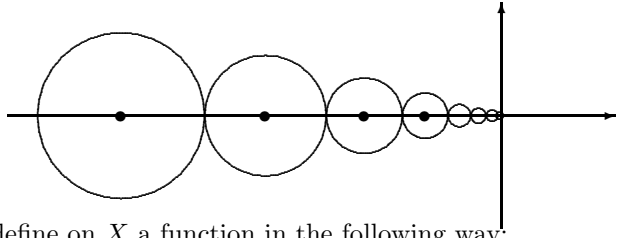
\begin{center}\setlength{\unitlength}{1mm}
\begin{picture}(60,30)
\put(-25,10){\vector(1,0){80}} \put(40,-5){\vector(0,1){30}}
\put(-10,10){\bigcircle{22}} \put(-11,9){$\bullet$}
\put(9,10){\bigcircle{16}}\put(8,9){$\bullet$}
\put(22,10){\bigcircle{10}}\put(21,9){$\bullet$}
\put(30,10){\bigcircle{6}}\put(29,9){$\bullet$}
\put(34.5,10){\bigcircle{3}} \put(37,10){\bigcircle{2}}
\put(38.8,10){\bigcircle{1.5}} \put(39.8,10){\bigcircle{1}}
\end{picture}
\end{center}
We define on $X$ a function in the following way:
$$ f(x,y)=\left\{
\begin{array}{rlll}
1&\text{if $(x,y)\in B_i$}& i=2k+1&k\in\Z,\\[4pt]
0&\text{if $(x,y)\in B_i$}& i=2k&k\in\Z.
\end{array}
\right.
$$
The constant function $g=0$ is clearly and upper gradient of $f$,
and so $f\in N^{1,\infty}(X)$. But, there is no continuous
representative for the function $f$. Thus, in particular, $f$ does
not admit a representative in $ D^{\infty}$.
\end{example}
In the following, we will look for conditions under which the
Sobolev spaces $M^{1,\infty}(X)$ and $N^{1,\infty}(X)$ coincide. In
particular, this will give us the equality of all the spaces in the
chain $(*)$ above. For that, we need some preliminary terminology
and results.

\begin{defn}\em
We say that a measure $\mu$ on $X$ is \em doubling \em if there is a
positive constant $C_{\mu}$ such that
$$
0<\mu(B(x,2r))\leq C_{\mu}\,\mu(B(x,r))<\infty,
$$
for each $x\in X$ and $r>0$. Here $B(x,r)$ denotes the open ball of
center $x$ and radius $r>0$.
\end{defn}

\begin{defn}  \em
Let $1\leq p<\infty$. We say that $(X,d,\mu)$ supports a \em weak
$p$-Poincaré inequality \em if there exist constants $C_p>0$ and
$\lambda\geq1$ such that for every Borel measurable function
$u:X\rightarrow\R$ and every upper gradient
$g:X\rightarrow[0,\infty]$ of $u$, the pair $(u,g)$ satisfies the
inequality
$$
\jint_{B(x,r)}|u-u_{B(x,r)}|\,d\mu\leq C_p\,r\Big(\jint_{B(x,\lambda
r)}g^p d\mu\Big)^{1/p}
$$
for each $B(x,r)\subset X$.

Here for arbitrary $A\subset X$ with $0<\mu(A)<\infty$ we write
$$
\jint_A f=\frac{1}{\mu(A)}\int_A f d\mu.
$$
\end{defn}
The Poincaré inequality creates a link between the measure, the
metric and the gradient and it provides a way to pass from the
infinitesimal information which gives the gradient to larger scales.
Metric spaces with doubling measure and Poincaré inequality admit
first order differential calculus akin to that in Euclidean spaces.
See \cite{Amb}, \cite{hei} or \cite{He2} for further information
about these topics.

The proof of the next result is strongly inspired in Proposition
$3.2$ in \cite{JJRRS}. However, we include all the details because
of the technical differences, which at certain points become quite
subtle.

\begin{thm}\label{Shamun}
Let $X$ be a complete metric space that supports a doubling Borel
measure $\mu$ which is non-trivial and finite on balls and suppose
that $X$ supports a weak $p$-Poincar\'e inequality for some $1\leq
p<\infty$. Let $\rho\in L^{\infty}(X)$ such that $0\leq\rho$. Then,
there exists a set $F\subset X$ of measure $0$ and a constant $K>0$
(depending only on $X$) such that for all $x,y\in X\setminus F$
there exist a rectifiable curve $\gamma$ such that
$\int_{\gamma}\rho<+\infty$ and $\ell(\gamma)\leq Kd(x,y)$.
\end{thm}
\begin{proof}
We may assume that $0<\|\rho\|_{L^{\infty}}\leq 1$. Indeed, in other
case, we could take
$\widetilde{\rho}=\rho/(1+\|\rho\|_{L^{\infty}})$. Let $E=\{x\in X:\
\rho(x)>\|\rho\|_{L^{\infty}}\}$, which is a set of measure zero. By
Theorem $2.2$ in \cite{hei}, there exists a constant $C$ depending
only on the doubling constant $C_\mu$ of $X$ such that for each
$f\in L^1(X)$ and for all $t>0$
$$
\mu(\{M(f)>t\})\leq \frac{C}{t}\int_X|f|d\mu
$$
Recall that $M(f)(x)=\sup_{r>0}\{\jintb_{B(x,r)} |f|d\mu$\}.

For each $n\geq 1$ we can choose $V_n$ be an open set such that
$E\subset V_n$ and $\mu(V_n)\leq\big(\frac{1}{n2^n}\big)^p$ (see
Theorem $1.10$ in \cite{Mat}). Note that $E\subseteq\bigcap_{n\geq
1}V_n=E_0$ and $\mu(E_0)=\mu(E)=0$.

Next, consider the family of functions
$$
\rho_n=\|\rho\|_{L^{\infty}}+\sum_{m\geq n}\chi_{V_m}
$$
and the function $\rho_0$ given by the formula
$$
\rho_0(x)=\left\{
\begin{array}{cc}
\|\rho\|_{L^{\infty}}&\text{ if $x\in X\setminus E_0$},\\[4pt]
+\infty&\text{ otherwise.}
\end{array}
\right.
$$
We have the following properties:
\begin{itemize}
\item[(i)] $\rho_n|_{X\setminus V_n}\equiv\|\rho\|_{L^{\infty}}$.
\item[(ii)] $\rho\leq\rho_0\leq \rho_m\leq \rho_n$ if $n\leq m$.
\item[(iii)] $\rho_n|_{E_0}\equiv+\infty$.
\item[(iv)] $\rho_n\in L^p(X)$ is lower semicontinuous; in fact $\|\rho_n-\|\rho\|_{L^\infty}\|_{L^p}\leq\frac{1}{n}$.

Indeed, since each of the sets $V_m$ are open then the functions
$\chi_{V_m}$ are lower semicontinuous (see Proposition $7.11$ in
\cite{fol}) and so once we check that
$\|\rho_n-\|\rho\|_{L^\infty}\|_{L^p}\leq\frac{1}{n}$, we will be
done. For that, is is enough to prove that $\sum_{m\geq
n}\|\chi_{V_m}\|_{L_p}\leq\frac{1}{n}$, which follows from the
formula
$$
\sum_{m\geq n}\|\chi_{V_m}\|_{L_p}=\sum_{m\geq
n}(\mu(V_n))^{1/p}=\sum_{m\geq
n}\frac{1}{m2^m}\leq\frac{1}{n}\sum_{m\geq
n}\frac{1}{2^m}\leq\frac{1}{n}.
$$

\item[(v)] $\mu(\{M((\rho_n-\|\rho\|_{L^{\infty}})^p)> 1\})\leq \frac{C}{n^{p}}$.

Indeed, as we have seen above
\begin{equation*}
\begin{split}
\mu(\{M((\rho_n-\|\rho\|_{L^{\infty}})^p)> 1\})\leq &
\,\frac{C}{1}\int_X|\rho_n-\|\rho\|_{L^{\infty}}|^p\\=&
\,C\|\rho_n-\|\rho\|_{L^\infty}\|_{L^p}^p<C\frac{1}{n^p}.
\end{split}
\end{equation*}
\end{itemize}
For each $n\geq 1$ consider the set
$$
S_n=\{x\in X:\ M((\rho_n-\|\rho\|_{L^{\infty}})^p)(x)\leq 1\}
$$
We claim that: \em $S_n\subset S_m$ if $n\leq m$ and
$F=X\setminus\bigcup_{n\geq1}S_n$ has measure $0$\em.

Indeed, if $n\leq m$, we have that $0\leq
\rho_m-\|\rho\|_{L^\infty}\leq\rho_n-\|\rho\|_{L^\infty}$ and so
$$
0\leq
(\rho_m-\|\rho\|_{L^\infty})^p\leq(\rho_n-\|\rho\|_{L^\infty})^p;
$$
hence $S_n\subset S_m$. On the other hand by (v) above, we have
$\mu(X\setminus S_n)\leq \frac{C}{n^{p}}$. Thus,
\begin{equation*}
\begin{split}
0\leq\mu(F)=&\mu\Big(X\setminus\bigcup_{n\geq1}S_n\Big)=\mu\Big(\bigcap_{n\geq
1}(X\setminus S_n)\Big)=\lim_{n\to\infty}\mu(X\setminus S_n)\leq
\lim_{n\to\infty}\frac{C}{n^{p}}=0.
\end{split}
\end{equation*}

After all this preparatory work, our aim is to prove that there
exists a constant $K>0$ depending only on $X$ such that for all
$x,y\in X\setminus F$ there exist a rectifiable curve $\gamma$ such
that $\int_{\gamma}\rho<+\infty$ and $\ell(\gamma)\leq K d(x,y)$.
The constant $K$ will be constructed along the remainder of the
proof. In what follows let $m_0$ be the smallest integer for which
$S_{m_0}\neq\emptyset$. Fix $n\geq m_0$ and a point $x_0\in
S_{n}\subset X\setminus F$. As one can check straightforwardly, it
is enough to prove that for each $x\in S_n$ there exists a
rectifiable curve $\gamma$ such that $\int_{\gamma}\rho<+\infty$ and
$\ell(\gamma)\leq K d(x,y)$, where the constant $K$ depends only on
$X$ and not on $x_0$ or $n$.

For our purposes, we define the set $\varGamma_{xy}$ as the set of
all the rectifiable curves connecting $x$ and $y$. Since a complete
metric space $X$ supporting a doubling measure and a weak
$p-$Poincaré inequality is quasi-convex (see Theorem $17.1$ in
\cite{Che}), it is clear that $\varGamma_{xy}$ is nonempty. We
define the function
$$
u_n(x)=\inf\Big\{\ell(\gamma)+\int_{\gamma}\rho_n:\
\gamma\in\varGamma_{x_0x}\Big\}.
$$
Note that $u_n(x_0)=0$. We will prove that in $S_n$ the function
$u_n$ is bounded by a Lipschitz function $v_n$ with a constant $K_0$
which depends only on $X$ and $\|\rho\|_{L^\infty}$ (and not on
$x_0$ nor $n$) such that $v_n(x_0)=0$. Assume this for a moment. We
have
$$
0\leq u_n(x)=u_n(x)-u_n(x_0)\leq v_n(x)-v_n(x_0)\leq
K_0d(x,x_0)<(K_0+1)d(x,x_0).
$$
Thus, there exists a rectifiable curve $\gamma\in\varGamma_{x_0x}$
such that
$$
\ell(\gamma)+\int_{\gamma}\rho\leq
\ell(\gamma)+\int_{\gamma}\rho_n\leq (K_0+1)d(x,x_0).
$$
Hence, taking $K=K_0+1$, we will have
$$
\ell(\gamma)\leq Kd(x,x_0)\qquad\text{and}\qquad
\int_{\gamma}\rho<+\infty,
$$
as we wanted.

Therefore, consider the functions $u_{n,k}:X\to\R$ given by
$$
u_{n,k}=\inf\Big\{\ell(\gamma)+\int_{\gamma}\rho_{n,k}:\
\gamma\in\varGamma_{x_0x}\Big\}
$$
where $\rho_{n,k}=\min\{\rho_n,k\}$ which is a lower semicontinuous
function. Let us see that the functions $u_{n,k}$ are Lipschitz for
each $k\geq 1$ (and in particular continuous) and that
$\rho_{n,k}+1\leq \rho_n+1$ are upper gradients for $u_{n,k}$. Since
$X$ is quasi-convex, it follows that $u_{n,k}(x)<+\infty$ for all
$x\in X$.

Indeed, let $y,z\in X$, $C_q$ the constant of quasi-convexity for
$X$ and $\veps>0$. We may assume that $u_{n,k}(z)\geq u_{n,k}(y)$.
Let $\gamma_y\in\varGamma_{x_0y}$ be such that
$$
u_{n,k}(y)\geq \ell(\gamma_y)+\int_{\gamma_y}\rho_{n,k}-\veps.
$$
On the other hand, for each rectifiable curve
$\gamma_{yz}\in\varGamma_{yz}$, we have
$$
u_{n,k}(z)\leq\ell(\gamma_y\cup\gamma_{yz})+\int_{\gamma_y\cup\gamma_{yz}}\rho_{n,k},
$$
and so
$$
|u_{n,k}(z)-u_{n,k}(y)|=u_{n,k}(z)-u_{n,k}(y)\leq\ell(\gamma_{yz})+\int_{\gamma_{yz}}\rho_{n,k}=\int_{\gamma_{yz}}(\rho_{n,k}+1).
$$
Thus, $\rho_{n,k}+1$ is an upper gradient for $u_{n,k}$. In
particular, if $\ell(\gamma_{zy})\leq C_qd(z,y)$, we deduce that
$$
|u_{n,k}(z)-u_{n,k}(y)|\leq (k+1)\ell(\gamma_{zy})\leq
C_q(k+1)d(z,y)
$$
and so $u_{n,k}$ is a $C_q(k+1)$-Lipschitz function. Our purpose now
is to prove that the restriction to $S_n$ of each function $u_{n,k}$
is a Lipschitz function on $S_n$ with respect to a constant $K_0$
which depends only on $X$. Fix $y,z\in S_n$ . For each $i\in \Z$,
define $B_i=B(z, 2^{-i},d(z,y))$ if $i\geq1$, $B_0=B(z,2d(z,y))$,
and $B_i =B(y,2^id(z,y))$ if $i\leq-1$. To simplify notation we
write $\lambda B(x,r)=B(x,\lambda  r)$. In the first inequality of
the following estimation we use the fact that, since $u_{n,k}$ is
continuous, all points of $X$ are Lebesgue points of $u_{n,k}$.
Using the weak $p$-Poincar\'e inequality and the doubling condition
we get the third inequality. From the Minkowski inequality we deduce
the fifth while the last one follows from the definition of $S_n$:
\begin{equation*}
\begin{split}
|&u_{n,k}(z)-u_{n,k}(y)|\leq\sum_{i\in\Z}\Big|\jint_{B_i}u_{n,k}d\mu-\jint_{B_{i+1}}u_{n,k}d\mu\Big|\\
&\stackrel{(*)}{\leq}\sum_{i\in\Z}\frac{1}{\mu(B_i)}\int_{B_i}\Big|u_{n,k}-\jint_{B_{i+1}}u_{n,k}d\mu\Big|d\mu\\
&\leq C_{\mu}C_pd(z,y)\sum_{i\in\Z}2^{-|i|}\Big(\frac{1}{\mu(\lambda  B_i)}\int_{\lambda  B_i}(\rho_{n,k}+1)^p\Big)^{1/p}\\
&\leq C_{\mu}C_pd(z,y)\sum_{i\in\Z}2^{-|i|}\Big(\frac{1}{\mu(\lambda  B_i)}\int_{\lambda  B_i}((\rho_{n,k}-\|\rho\|_{L^{\infty}})+\|\rho\|_{L^\infty}+1)^p\Big)^{1/p}\\
&\leq C_{\mu}C_pd(z,y)\sum_{i\in\Z}2^{-|i|}\Big(\|\rho\|_{L^\infty}+1+\Big(\frac{1}{\mu(\lambda  B_i)}\int_{\lambda  B_i}(\rho_{n,k}-\|\rho\|_{L^{\infty}})^p\Big)\Big)^{1/p}\Big)\\
&\leq 3\, C_{\mu}C_pd(z,y)\sum_{i\in\Z}2^{-|i|}\leq K_0d(z,y)
\end{split}
\end{equation*}
where $K_0=9C_{\mu}C_p$ is a constant that depends only on $X$.
Recall that $C_{\mu}$ is the doubling constant and $C_p$ is the
constant which appears in the weak $p-$Poincaré inequality. Let us
see with more detail inequality $(*)$. If $i>0$, we have that
\begin{equation*}
\begin{split}
\Big|\jint_{B_i}u_{n,k}d\mu-\jint_{B_{i+1}}u_{n,k}d\mu\Big|\leq&\,\,\frac{1}{\mu(B_{i+1})}\Big|\int_{B_{i+1}}\Big(u_{n,k}-\jint_{B_{i}}u_{n,k}\,d\mu\Big)d\mu\Big|\\
\leq&\,\,\frac{\mu(B_i)}{\mu(B_i)}\frac{1}{\mu(B_{i+1})}\Big|\int_{B_{i}}\Big(u_{n,k}-\jint_{B_{i}}u_{n,k}\,d\mu\Big)d\mu\Big|\\
\leq&\,\,\frac{C_{\mu}}{\mu(B_{i})}\Big|\int_{B_{i}}\Big(u_{n,k}-\jint_{B_{i}}u_{n,k}\,d\mu\Big)d\mu\Big|.
\end{split}
\end{equation*}
We have used that $B_{i+1}\subset B_i$ for $i>0$ and that $\mu$ is a
doubling measure and so $\mu(2B_{i+1})=\mu(B_i)\leq
C_{\mu}\,\mu(B_{i+1})$. The cases $i<0$ and $i=0$ are similar.

Thus, the restriction of $u_{n,k}$ to $S_n$ is a $K_0$-Lipschitz
function for all $k\geq 1$. Note that $u_{n,k}\leq u_{n,k+1}$ and
therefore we may define
$$
v_n(x)=\sup_k\{u_{n,k}(x)\}=\lim_{k\to\infty}u_{n,k}(x).
$$
Whence $v_n$ is a $K_0$-Lipschitz function on $S_n$. Since
$v(x_0)=0$ and $x_0\in S_m$ when $m\geq m_0$ we have that
$v(x)<\infty$ and so, it is enough to check that $u_n(x)\leq v_n(x)$
for $x\in S_n$. Now, fix $x\in S_n$. For each $k\geq 1$ there is
$\gamma_k\in\varGamma_{x_0x}$ such that
$$
\ell(\gamma_k)+\int_{\gamma_k}\rho_{n,k}\leq
u_{n,k}(x)+\frac{1}{k}\leq K_0d(x,x_0)+\frac{1}{k}.
$$
In particular, $\ell(\gamma_k)\leq K_0d(x,x_0)+1:=M$ for every
$k\geq1$ and so, by reparametrization, we may assume that $\gamma_k$
is an $M$-Lipschitz function and
$\gamma_k:[0,1]\to\overline{B(x_0,M)}$ for all $k\geq 1$. Since X is
complete and doubling, and therefore closed balls are compact, we
are in a position to use the Ascoli-Arzela theorem to obtain a
subsequence $\{\gamma_k\}_k$ (which we denote again by
$\{\gamma_k\}_k$ to simplify notation) and $\gamma:[0,1]\to X$ such
that $\gamma_k\to\gamma$ uniformly. For each $k_0$, the function
$1+\rho_{n,k_0}$ is lower semicontinuous,
 and therefore by Lemma $2.2$ in \cite{JJRRS} and the fact that $\{\rho_{n,k}\}$ is an increasing sequence of functions,
 we have
\begin{equation*}
\begin{split}
\ell(\gamma)+\int_{\gamma}\rho_{n,k_0}=&\int_{\gamma}(1+\rho_{n,k_0})\leq\lim\inf_{k\to\infty}\int_{\gamma_k}(1+\rho_{n,k_0})\leq\lim\inf_{k\to\infty}\int_{\gamma_k}(1+\rho_{n,k}).
\end{split}
\end{equation*}
Using the monotone convergence theorem on the left hand side and
letting $k_0$ tend to infinity yields
$$
\ell(\gamma)+\int_{\gamma}\rho_{n}\leq\lim\inf_{k\to\infty}\int_{\gamma_k}(1+\rho_{n,k}).
$$
Since $\gamma\in\varGamma_{x_0x}$ we have
\begin{equation*}
\begin{split}
u_n(x)\leq\ell(\gamma)+\int_{\gamma}\rho_{n}\leq&\lim\inf_{k\to\infty}\int_{\gamma_k}(1+\rho_{n,k})
\\\leq&\lim\inf_{k\to\infty}\Big(u_{n,k}(x)+\frac{1}{k}\Big)\leq
v_n(x),
\end{split}
\end{equation*}
and that completes the proof.
\end{proof}
\begin{remark}\em
In Theorem \ref{Shamun} we can change the hypothesis of completeness
for the space $X$ by local compactness. The proof is analogous to
the one of Theorem $1.6$ in \cite{JJRRS}, and we do not include the
details.
\end{remark}

\begin{cor}\label{NLip}
Let $X$ be a complete metric space that supports a doubling Borel
measure $\mu$ which is non-trivial and finite on balls. If $X$
supports a weak $p$-Poincar\'e inequality for $1\leq p<\infty$, then
$\LIP^{\infty}(X)=M^{1,\infty}(X)=N^{1,\infty}(X)$ with equivalent
norms.
\end{cor}

\begin{proof}
If $f\in N^{1,\infty}(X)$, then there exists an $\infty-$weak upper
gradient $g\in L^\infty(X)$ of $f$. We denote $\Gamma_1$ the family
of curves for which $g$ is not an upper gradient for $f$. Note that
$\Mod_{\infty}\Gamma_1$=0. By Lemma \ref{modinfzero} there exists a
Borel measurable function $0\leq\varrho\in L^{\infty}(X)$ such that,
$\int_{\gamma}\varrho=+\infty\,\,\text{ for every
}\gamma\in\Gamma_1$ and $\|\varrho\|_{L^{\infty}}=0$.
 Consider $\rho_0=g+\varrho\in L^{\infty}(X)$
which is an upper gradient of $f$ and satisfies that
$\|\rho_0\|_{L^\infty}=\|g\|_{L^\infty}$. Note that
$\int_{\gamma}\rho_0=+\infty$ for all $\gamma\in\Gamma_1$ and that
by Lemma \ref{modinfzero} the family of curves
$\Gamma_2=\{\gamma\in\Upsilon:\int_{\gamma}\rho_0=+\infty\}$ has
$\infty-$modulus zero. Finally, consider the set $\{x\in
X:g(x)+\varrho(x)\geq\|\rho_0\|_{L^\infty}\}$ and define
$$
\rho(x)= \left\{
\begin{array}{ll}
\|\rho_0\|_{L^\infty}&\text{if $x\in X\backslash E,$}\\[5pt]
+\infty&\text{if $x\in E$.}
\end{array}
\right.
$$
Then $\rho$ is an upper gradient of $f$  and it satisfies that
$\|\rho\|_{L^{\infty}}=\|\rho_0\|_{L^\infty}=\|g\|_{L^\infty}$. Note
that if $\int_{\gamma}\rho<+\infty$, then the set
$\gamma^{-1}(+\infty)$ has measure zero in the domain of $\gamma$
(because otherwise $\int_{\gamma}\rho=+\infty$). Thus, if
$\int_{\gamma}\rho<+\infty$, we have in particular that
$\int_{\gamma}\rho=\|\rho\|_{L^{\infty}}\,\ell(\gamma)$. By Theorem
\ref{Shamun} there exists a set $F\subset X$ of measure $0$ and a
constant $K>0$ (depending only on $X$) such that for all $x,y\in
X\setminus F$ there exist a rectifiable curve $\gamma$ such that
$\int_{\gamma}\rho<+\infty$ and $\ell(\gamma)\leq Kd(x,y)$. Let now
$x,y\in X\backslash F$ and $\gamma$ be a rectifiable curve
satisfying the precedent conditions. Then
$$
|f(x)-f(y)|\leq\int_{\gamma}\rho\stackrel{(*)}{=}\|\rho\|_{L^{\infty}}\ell(\gamma)\leq
\|\rho\|_{L^{\infty}}K d(x,y).
$$
Then $f$ is $\|\rho\|_{L^{\infty}}K-$Lipschitz a.e. Thus,
$\LIP^{\infty}(X)=M^{1,\infty}(X)=N^{1,\infty}(X)$.
\end{proof}

\begin{remark}\em Note that if we would have chosen as upper gradient $\rho_0$
instead of $\rho$, the inequality $(*)$ might not be necessary true.
To see this, it is enough to define a function which is zero a.e.
and constant but finite on a set of zero measure.
\end{remark}

Our purpose now is to see under which conditions the spaces
$D^{\infty}(X)$ and $N^{1,\infty}(X)$ coincide. For that, we need
first to use the local version of the weak $p$-Poincaré inequality
(see for example Definition $4.2.17$ in \cite{Sha}).

\begin{defn}  \em
Let $1\leq p<\infty$. We say that $(X,d,\mu)$ supports a \em local
weak $p$-Poincaré inequality \em with constant $C_p$ if for every
$x\in X$, there exists a neighborhood $U_x$ of $x$  and
$\lambda\geq1$ such that whenever $B$ is a ball in $X$ such that
$\lambda B$ is contained in $U_x$, and $u$ is an integrable function
on $\lambda B$ with $g$ as its upper gradient in $\lambda B $, then
$$
\jint_{B(x,r)}|u-u_{B(x,r)}|\,d\mu\leq C_pr\Big(\jint_{B(x,\lambda
r)}g^p d\mu\Big)^{1/p}.
$$
\end{defn}
\vspace{2mm}
\begin{cor}\label{N=D}
Let $X$ be a complete metric space that supports a doubling Borel
measure $\mu$ which is non-trivial and finite on balls. If $X$
supports a local weak $p$-Poincar\'e inequality for $1\leq
p<\infty$. Then $N^{1,\infty}(X)=D^{\infty}(X)$ with equivalent
norms.
\end{cor}

\begin{proof}
If $f\in N^{1,\infty}(X)$, then there exists an $\infty-$weak upper
gradient $g\in L^\infty(X)$ of $f$. We construct in the same way as
in Corollary \ref{NLip} an upper gradient $\rho$ of $f$ which
satisfies $\|\rho\|_{L^{\infty}}=\|g\|_{L^\infty}$,
$\int_{\gamma}\rho\geq |f(\gamma(0))-f(\gamma(L))|$ for all
$\gamma\in\Upsilon$ and $\int_{\gamma}\rho=
\|\rho\|_{L^{\infty}}\ell(\gamma)$ for all $\gamma\in\Upsilon$ such
that $\int_{\gamma}\rho<+\infty$. Fix $x\in X$. Using a local
version of Theorem \ref{Shamun} we obtain that there exists a
neighborhood $U^x$ and a constant $K>0$ (depending only on $X$) such
that for almost every $z,y\in U^x$, there exist a rectifiable curve
$\gamma$ connecting $z$ and $y$ such that
$\int_{\gamma}\rho<+\infty$ and $\ell(\gamma)\leq Kd(z,y)$. Let now
$y\in U^x$ and $\gamma$ a rectifiable curve satisfying the precedent
conditions. Then
$$
|f(x)-f(y)|\leq\int_{\gamma}\rho=\|\rho\|_{L^{\infty}}\ell(\gamma)\leq
\|\rho\|_{L^{\infty}}K d(x,y).
$$

Under the hypothesis of the corollary it can be easily checked that
$f$ is continuous on $X$ and so, there is no obstruction to take the
superior limit
$$\limsup_{y\to x}\frac{|f(x)-f(y)|}{d(x,y)}.$$ Thus, we deduce that $\Lip f(x)\leq
K\|\rho\|_{L^{\infty}}$. Since this is true for each $x\in X$, we
have $\|\Lip f\|_{\infty}\leq K\|\rho\|_{L^{\infty}}<+\infty$ and we
conclude that $f\in D^{\infty}(X)$.
\end{proof}

Observe that under the hypothesis of Corollary \ref{N=D} we have
that $X$ is a locally radially quasiconvex metric space. We see
throughout a very simple example that in general there exist metric
spaces $X$ for which the following holds:
$$
\LIP^{\infty}(X)=M^{1,\infty}(X)\varsubsetneq D^{\infty}(X)=
N^{1,\infty}(X).
$$
Indeed, consider the metric space $(X,d,\lambda)$ where
$X=\C\backslash\{\rm Re(z)\geq 0, |Im(z)|\leq 1/2\}$, $d$ is the
metric induced by the Euclidean one and $\lambda$ denotes the
Lebesgue measure. Since $X$ is a complete metric space that supports
a doubling measure and a local weak $p$-Poincar\'e inequality for
any $1\leq p<\infty$, by Corollary \ref{N=D}, we have that
$D^{\infty}(X)=N^{1,\infty}(X)$. Let $f(z)=\arg(z)$, for each $z\in
X$. One can check that $f\in D^{\infty}(X)=N^{1,\infty}(X)$.
However, $f\notin \LIP^{\infty}(X)$, and so $
\LIP^{\infty}(X)\varsubsetneq D^{\infty}(X)= N^{1,\infty}(X). $

\centerline{\sc Acknowledgements}

It is a great pleasure to thank Professors Jose F. Fernando and M.
Isabel Garrido for many valuable conversations concerning this
paper.

\end{example}

\begin{thebibliography}{999999}

\addcontentsline{toc}{chapter}{Bibliograf\'\i a}
\renewcommand{\baselinestretch}{1}
\bibitem[Am]{Amb}L. Ambrosio, P. Tilli: Topics on Analysis in
Metric Spaces. Oxford Lecture Series in Mathematics and its
Applications {\bf 25}. \em Oxford University Press, Oxford\em,
(2004).
\bibitem[BRZ]{BRZ}  Z. M. Balogh, K. Rogovin, T. Zürcher: The Stepanov Differentiability Theorem in Metric Measure
Spaces. \em J. Geom. Anal. \em {\bf14} no. 3, (2004), 405--422.
\bibitem[Ch]{Che} J. Cheeger: Differentiability of Lipschitz
Functions on metric measure spaces. \em Geom. Funct. Anal. \em
\textbf{9} (1999), 428--517.
\bibitem[F]{fol} G.B. Folland: Real Analysis, Modern Techniques
and Their Applications. Pure and Applied Mathematics (1999).
\bibitem[Fu]{Fu} B. Fuglede : Extremal length and functional
completion. \em Acta. Math. \em {\bf98} (1957), 171--219.
\bibitem[GJ1]{GaJa1}M. I. Garrido, J. A. Jaramillo: A Banach-Stone
Theorem for Uniformly Continuous Functions. \em Monatshefte für
mathematik. \em {\bf131} (2000), 189--192.
\bibitem[GJ2]{GaJa2}M. I. Garrido, J. A. Jaramillo: Homomorphism on
Function Lattices. \em Monatshefte für mathematik. \em {\bf141}
(2004), 127--146.
\bibitem[GJ3]{GaJa3} M. I. Garrido, J. A. Jaramillo: Lipschitz-type
functions on metric spaces. \em J. Math. Anal and Appl. \em {\bf340}
(2008), 282--290.
\bibitem[GiJe]{GiJe} L. Gillman, J. Jerison: Rings of Continuous Functions. Springer-Verlag, New-York (1976).
\bibitem[Ha1]{haj} P. Haj{\l}asz: Sobolev spaces on metric-measure
spaces. \em Contemp. Math. \em {\bf338} (2003), 173--218.
\bibitem[Ha2]{haj1} P. Haj{\l}asz: Sobolev spaces on an arbitrary metric
space. \em Potential Anal. \em {\bf5} (1996), 403--415.
\bibitem[He1]{hei} J. Heinonen: Lectures on Analysis on Metric Spaces. Springer (2001).
\bibitem[He2]{He2} J. Heinonen: Nonsmooth calculus. \em Bull. Amer. Math. Soc. \em {\bf44} (2007),
163--232.
\bibitem[HK]{HK} J. Heinonen, P. Koskela: Quasiconformal maps on
metric spaces with controlled geometry. \em Acta Math. \em {\bf181}
(1998), 1--61.
\bibitem[I]{I} J. R. Isbell: Algebras of uniformly continuous functions. \em Ann. Math. \em {\bf68}
(1958), 96--125.
\bibitem[JJRRS]{JJRRS} E. Järvenpää, M. Järvenpää, N. Shanmugalingam K. Rogovin, and S. Rogovin: Measurability of equivalence classes and MEC$_p$-property in metric
spaces. \em Rev. Mat. Iberoamericana \em {\bf23} (2007), 811--830.
\bibitem[KMc]{KosMac} P. Koskela, P. MacManus: Quasiconformal
mappings and Sobolev spaces. \em Studia Math. \em  {\bf131} (1998),
1--17.
\bibitem[K]{Kei}S. Keith: A differentiable structure for metric
measure spaces. \em Adv. Math. \em {\bf183} (2004), 271--315.
\bibitem[Ma]{Ma}V. Magnani: Elements of Geometric Measure Theory on sub-Riemannian groups, (2002). Dissertation. Scuola Normale Superiore di Pisa.
\bibitem[M]{Mat} P. Mattila: Geometry of Sets and Measures in Euclidean Spaces: Fractals and rectifiability. \em Cambridge studies in Advance Mathematics\em, Cambridge University Press {\bf44} (1995).
\bibitem[Me]{M} R. E. Megginson, R. E.: An introduction to Banach Space Theory.
               Graduate Texts in Mathematics, {\bf183}. \em Springer-Verlag, New York \em, 1998.
\bibitem[S]{Sem} S. Semmes: Some Novel Types of Fractal Geometry.
Oxford Science Publications (2001).
\bibitem[Sh1]{Sha}N. Shanmugalingam: ``Newtonian Spaces: An extension
of Sobolev spaces to Metric Measure Spaces" Ph. D. Thesis,
University of Michigan (1999), http: math.uc.edu/~nages/papers.html.
\bibitem[Sh2]{Sh1}N. Shanmugalingam: Newtonian Spaces: An
extension of Sobolev spaces to Metric Measure Spaces. \em Rev. Mat.
Iberoamericana\em, {\bf16} (2000), 243--279.
\bibitem[W]{Wea} N. Weaver: Lipschitz Algebras. Singapore: World Scientific (1999).
\end{thebibliography}
\end{document}